\documentclass[12pt,a4]{amsart}
\usepackage[a4paper, left=28mm, right=28mm, top=28mm, bottom=34mm]{geometry}
\allowdisplaybreaks

\usepackage{amssymb
,amsthm
,amsmath
,amscd
,mathtools	
,mathdots
,leftidx
,color
}
\usepackage{hyperref}
\usepackage[all]{xy}
\usepackage{url}
\usepackage[dvipdfmx]{graphicx}
\AtBeginDocument{%
   \def\MR#1{}
}

\newcommand{\GL}{\mathrm{GL}}

\newcommand{\SO}{\mathrm{SO}}
\newcommand{\algO}{\mathrm{O}}

\newcommand\PGL{\mathrm{PGL}}
\newcommand\hyp{irreducible symplectic }

\newcommand{\oo}{\mathcal{O}}

\newcommand\Z{\mathbb{Z}}

\newcommand\F{\mathbb{F}}
\newcommand\Q{\mathbb{Q}}
\newcommand\R{\mathbb{R}}
\newcommand\C{\mathbb{C}}

\newcommand\et{\textup{\'et}}

\DeclareMathOperator{\Aut}{Aut}

\newcommand\Isom{\textup{Isom}}

\newcommand\Gal{\mathrm{Gal}}

\newcommand\Pic{\mathop{\mathrm{Pic}}\nolimits}
\DeclareMathOperator{\disc}{disc}
\newcommand\Spec{\mathop{\mathrm{Spec}}}
\newcommand\Frac{\mathrm{Frac}}

\newcommand\Tw{\mathrm{Tw}}
\newcommand\spp{\mathrm{sp}}

\newcommand\nef{\mathrm{Nef}}

\newcommand\la{\mathcal{L}}

\newcommand\W{\mathrm{W}}

\newcommand\rk{\mathrm{rk}}

\newcommand\mv{\mathcal{MV}}
\newcommand\amp{\mathrm{Amp}}
\newcommand\pos{\mathcal{C}}
\DeclareMathOperator{\Bir}{Bir}

\newcommand\barX{X_{\overline{k}}}
\newcommand\sepX{X_{k_{s}}}
\newcommand\Pex{\mathcal{P}ex}
\newcommand\ba{\mathcal{BA}}
\DeclareMathOperator{\Ima}{Im}
\DeclareMathOperator{\chara}{char}
\DeclareMathOperator{\ord}{ord}

\newcommand\ch{\mathrm{ch}}
\newcommand\td{\mathrm{td}}

\theoremstyle{plain}
\newtheorem{theorem}[subsubsection]{Theorem}
\newtheorem{prop}[subsubsection]{Proposition}
\newtheorem{lemma}[subsubsection]{Lemma}

\newtheorem{corollary}[subsubsection]{Corollary}

\theoremstyle{definition}

\newtheorem{definition}[subsubsection]{Definition}
\newtheorem{remark}[subsubsection]{Remark}

\newtheorem*{lemma*}{Lemma}
\newtheorem*{prop*}{Proposition}
\newtheorem*{theorem*}{Theorem}
\newtheorem*{claim*}{Claim}
\newtheorem{definition*}{Definition}

\newtheorem*{ack}{Acknowledgements}

\begin{document}
\title{On the finiteness of twists of irreducible symplectic varieties}
\author[T.\ Takamatsu]{Teppei Takamatsu}


\address{Department of Mathematics (Hakubi Center), Graduate School of Science, Kyoto University, Kyoto 606-8502, Japan}
\email{teppeitakamatsu.math@gmail.com}

\maketitle

\begin{abstract}
Irreducible symplectic varieties are higher-dimensional analogues of K3 surfaces.
In this paper, we prove the finiteness of twists of \hyp varieties for a fixed finite field extension of characteristic $0$. 
The main ingredient of the proof is the cone conjecture for \hyp varieties, which was proved by Markman and Amerik--Verbitsky. As byproducts, we also discuss the cone conjecture over non-closed fields by Bright--Logan--van Luijk's method.
We also give an application to the finiteness of derived equivalent twists.
Moreover, we discuss the case of K3 surfaces or Enriques surfaces over fields of positive characteristic. 
\end{abstract}

\setcounter{section}{0}

\section{Introduction}
Classifying algebraic varieties is one of the main purposes of algebraic geometry.
Over an algebraically closed field, there are many such classification theorems.
However, over a general base field, such classification problems become difficult since there might be different twisted forms for a given variety.
In this paper, we study the finiteness of twists of varieties, especially for \hyp varieties.
Let $k$ be a field, and $k'$ a finite field extension of $k$.  
Let $X$ be a variety over $k$.
We denote the set of $k$-isomorphism classes of varieties over $k$ which are isomorphic to $X_{k'} \coloneqq X \times_{k} k'$ after the base change to $k'$ by $\Tw_{k'/k}(X)$.
In general, the set $\Tw_{k'/k}(X)$ is not necessarily finite even when $X$ is smooth and projective over $k$.
For example, if $X$ is a projective space $\mathbb{P}^{n}_{k}$ over a number field $k$ and $k'/k$ is a non-trivial extension, then one can prove that $\Tw_{k'/k} (X)$ is infinite by the Morita equivalence and the Chebotarev's density theorem. 
On the other hand, if $k'/k$ is a Galois extension and $X_{k'}$ has an automorphism group of finite order, then the Galois cohomological argument asserts that $\Tw_{k'/k}(X)$ is finite. In particular, $\Tw_{k'/k} (X)$ is finite if $X$ is general type varieties of characteristic $0$.
Therefore, this finiteness reflects certain finiteness of the automorphism group of a variety, and it seems to be most interesting when the Kodaira dimension of $X$ is $0$.
The main theorem of this paper is the following.
\begin{theorem}[Theorem \ref{fintwistK3}, Theorem \ref{hypkfintwist}]
\label{intromain}
Let $k'/k$ be a finite extension of fields, $X$ a variety over $k$.
Then $\Tw_{k'/k} (X)$ is a finite set in the following cases.
\begin{enumerate}
\item
$\chara k \neq 2$, and $X$ is a K3 surface over $k$.
\item
$\chara k = 2$, and $X$ is a non-supersingular K3 surface over $k$
\item
$\chara k = 0$, and $X$ is an \hyp variety over $k$.
\end{enumerate} 
\end{theorem}
Moreover, in the above cases, we also prove that $\# \Tw_{k'/k} (X)$ is bounded by a constant which depends only on $[k'\colon k]$, the isometry class of geometric N\'{e}ron--Severi lattice, and the deformation class of  $X$ (see Theorem \ref{effective} for precise statements).

Theorem \ref{intromain} is a generalization of Cattaneo--Fu's work (\cite{Cattaneo2019}) on the finiteness of real forms of \hyp varieties. 
By the example of projective spaces, we can say that the finiteness of $\Tw_{k'/k}$ is more interesting when $k$ is a global field rather than a local field.
Indeed, when $k$ is a global field, as in \cite[Corollary 6.2.1]{Takamatsu2020a}, such finiteness is one of the evidence of the finiteness of varieties with cohomologies of bounded ramification (cohomological formulation of the Shafarevich conjecture).

We also give such finiteness for Enriques surfaces of arbitrary characteristic, by using Liedtke's lifting result for Enriques surfaces \cite{Liedtke2015}.
This finiteness can be seen as one of the evidence of several finiteness properties of the automorphism group of Enriques surfaces of characteristic $2$.
\begin{theorem}[Theorem \ref{fintwistEnr}]
\label{introenr}
Let $k'/k$ be a finite extension of perfect fields of characteristic $2$, $X$ an Enriques surface over $k$.
Then $\Tw_{k'/k} (X)$ is a finite set.
\end{theorem}

Moreover, as an application, we give the following finiteness of derived equivalent twists. We recall that varieties are derived equivalent if there is a $k$-linear equivalence of triangulated categories between the bounded derived categories of coherent sheaves. 
We denote the set of $k$-isomorphism classes of varieties over $k$ which are derived equivalent to $X$ and isomorphic to $\barX$ after the base change to an algebraic closure $\overline{k}$ by $\Tw^{D} (X)$, which is a subset of $\Tw_{\overline{k}/k} (X)$. 
Note that, $\Tw_{\overline{k}/k} (X)$ is not finite in most cases.
\begin{corollary}[Theorem \ref{findertwist}]
\label{introcor}
In the case (1) or (2) in Theorem \ref{intromain}, the set $\Tw^{D}(X)$ is a finite set.
\end{corollary}

We also prove the same finiteness for all known \hyp varieties, i.e.\ $K3^{[n]}$-type, generalized Kummer type, $OG_{6}$-type, and $OG_{10}$-type (see Theorem \ref{findertwist} and Corollary \ref{corfindertwist} for precise statements).
Combining Corollary \ref{introcor} with \cite[Corollary 1.2]{Bridgeland2001} and \cite[Theorem 1.1]{Lieblich2015}, we have the following finiteness over non-closed fields.

\begin{corollary}
Let $X$ be a K3 surface over a field $k$ of characteristic $\neq 2$.
Then there exist only finitely many $k$-isomorphism classes of smooth projective varieties $Y$ over $k$ which are derived equivalent to $X$.
\end{corollary}

Now we give some comments on the proof of Theorem \ref{intromain} for \hyp varieties.
In Cattaneo--Fu's work (\cite{Cattaneo2019}), first, they study the Klein automorphism group of \hyp varieties by using the cone conjecture for \hyp varieties, and then they see the Galois cohomology directly.
In our situation, there is no direct analogue of the Klein automorphism group.
Thus to avoid the difficulty of the group structure and the Galois cohomological arguments, we take a different approach, though we use the cone conjecture for \hyp varieties too.
The main idea is reducing the problem to the finiteness of twists of quasi-polarized \hyp variety.
In this reduction, we need two steps. First, we should bound the polarization degree of \hyp variety. Second, we need the finiteness of the set of polarizations of bounded degree modulo automorphisms for \hyp varieties.
Both steps are related to the cone conjecture for \hyp varieties, proved by Markman \cite{Markman2011} and Amerik and Verbitsky \cite{Amerik2017}, \cite{Amerik2016} when the base field is algebraically closed (see also Markman and Yoshioka's work \cite{Markman2015}).
As byproducts, we also prove the following cone conjecture for \hyp varieties over non-closed fields of characteristic $0$.
\begin{theorem}[Theorem \ref{bircone}, Theorem \ref{automcone}]
\label{introcone}
Let $k$ be a field of characteristic $0$.
Let $X$ be an \hyp variety over $k$.
\begin{enumerate}
\item
The action of $\Bir (X)$ on $\mv_{X}^{+}$ admits a rational polyhedral fundamental domain. 
\item
The action of $\Aut (X)$ on $\nef_{X}^{+}$ admits a rational polyhedral fundamental domain.
\end{enumerate}
\end{theorem}

To prove Theorem \ref{introcone}, we follow the method given by Bright--Logan--van Luijk \cite{Bright2019}, where they proved the cone conjecture for K3 surfaces over non-closed fields of characteristic different from $2$.

The outline of this paper is as follows.
In Section \ref{Preliminaries}, we recall the definition of \hyp varieties, several cones of surfaces and \hyp varieties, and almost abelian groups. Moreover, we prove the finiteness of the Galois cohomology of an almost abelian group, which will be used for \hyp varieties of the geometric Picard rank $2$.
In Section \ref{Surfaces}, we prove the finiteness of twists for K3 surfaces and Enriques surfaces. In Section \ref{Hyperkahlervarieties}, we will prove the birational and automorphism cone conjectures for \hyp variety following the method by Bright--Logan--van Luijk, and prove the main theorem.
In Section \ref{Uniformbounds}, we argue the uniform boundedness of $\Tw_{k'/k} (X)$. 
In Section \ref{Derivedequivalent}, we recall definitions and generalities of derived equivalent varieties, and we prove the finiteness of $\Tw^{D}(X)$.

\subsection*{Conflict of interest}
The author declares no conflicts of interest associated with this manuscript.

\subsection*{Data availability}
Data sharing not applicable to this article as no datasets were generated or analysed during the current study.

\begin{ack}
The author is deeply grateful to his advisor Naoki Imai for deep encouragement and helpful advice. The author also would like to thank Tetsushi Ito for helpful comments on Lemma \ref{lemmamainrefined}.
Moreover, the author would like to thank Yohsuke Matsuzawa, Shou Yoshikawa, Kenta Hashizume, Alexei N. Skorobogatov, Yuki Yamamoto for helpful suggestions.
The author also thank the referee for constructive suggestions and reading the manuscript carefully.
The author was supported by JSPS KAKENHI Grant number JP19J22795.
\end{ack}

\section{Preliminaries}
\label{Preliminaries}
\subsection{Definitions and notations}
First, we review the definition of \hyp varieties.

\begin{definition}\label{defhyp}
\begin{enumerate}
\item
Let $k$ be a field of characteristic $0$.
Let $X$ be a smooth projective variety over $k$.
We say $X$ is an \emph{\hyp variety} if $X_{\overline{k}}$ is simply connected of even dimension $2n$ and there exists $\omega \in H^{0}(X, \Omega_{X/k}^{2})$ uniquely up to constant, such that $\omega^{n}$ vanishes nowhere. 
\item
Let $k$ be a subfield of $\C$.
Let $X$ be an \hyp variety over $k$. 
It is known that there exists a unique natural integral non-degenerate symmetric primitive bilinear pairing $q$ on $H^{2}(X_{\C},\Z)$, of signature $(3, b_{2}(X_{\C})-3)$ satisfying that there exists a positive rational number $\beta_{X}$ such that 
\begin{align*}
q(\alpha, \alpha)^{n} = \beta_{X} \int_{X} \alpha^{2n}
\end{align*}
for any $\alpha \in H^{2}(X_{\C},\Z)$,
and that $q(\alpha, \alpha)$ is positive for an ample class $\alpha \in H^{2}(X_{\C},\Z)$ (the second condition is automatic if $b_{2}(X_{\C}) \neq 6$).
We refer to $q$ as the \emph{Beauville--Bogomolov form}, and we denote this pairing $q$ by $(\ast, \ast)$ for short.
We note that $q$ is $\Aut(\C/k)$-invariant by the argument in \cite[Proposition 2.1.5]{Yang2019}.
\item
Let $k$ be a field of characteristic $0$.
Let $X$ be an \hyp variety over $k$.
Since $X$ is defined over a finitely generated subfield $k' \subset k$ (i.e. there exists an \hyp variety $Y$ over $k'$ such that $Y_{k} \simeq X$), 
by fixing an embedding $k' \subset \C$, we have an isomorphism
$H^{i}_{\et}(X_{\overline{k}},\Z_{\ell}) \simeq H^{i}(Y_{\C},\Z) \otimes \Z_{\ell}$.
Through this isomorphism, we can associate an integral non-degenerate symmetric bilinear pairing $q$ on $H^{2}_{\et}(X_{\overline{k}},\Z_{\ell})$.
By the argument in \cite[Proposition 2.1.5]{Yang2019} (see also \cite[Lemma 2.1.1]{Yang2019}), $q$ is independent of the choice of $k'$ and an embedding $k' \subset \C$, and $q$ is also $\Gal (\overline{k}/k)$-invariant.
We denote this pairing by $(\ast, \ast)$.
We also denote the $\Z_{\ell}$-valued symmetric bilinear pairing on $\Pic_{X/k} (k)$ which is given by the composition of the Chern character and $q$ by $(\ast, \ast)$.
Here, we denote the Picard functor of $X$ over $k$ by $\Pic_{X/k}$.
This pairing is actually $\Z$-valued as in \cite[Proposition 2.1.5]{Yang2019}.
\end{enumerate}
\end{definition}

\begin{remark}
Let $X_{1}, X_{2}$ be an \hyp variety over a field $k$ of characteristic $0$.
Let $f \colon X_{1} \dashrightarrow X_{2}$ be a birational map.
Since $K_{X_{i}}$ are trivial and $X_{i}$ are terminal, the map $f$ is pseudo-isomorphism (i.e.\, isomorphic in codimension $1$).
Indeed, if we take a smooth projective variety $Y$ over $k$ with birational morphisms $\pi_{1}\colon Y \rightarrow X_{1}$ and $\pi_{2}\colon Y \rightarrow X_{2}$, then the set of $\pi_{i}$-exceptional prime divisors are no other than the set of prime divisors $E \subset Y$ such that $\ord_{E}(K_{Y}-\pi_{i}^{\ast} (K_{X_{i}})) >0$, which do not depend on $i$.
Therefore, one can pull back a line bundle on $X_{2}$ and an element of $\Pic_{X/k}(k)$ via $f$.
\label{remsmall}
\end{remark}

Next, we recall the definition of several cones of surfaces or \hyp varieties.

\begin{definition}
Let $k$ be a field.
Let $X$ be a smooth projective variety over $k$.
Suppose that $X$ is a surface (resp.\ the characteristic of $k$ is $0$ and $X$ is an \hyp variety over $k$). 
Let $\Lambda_{X}$ be the free part of N\'{e}ron--Severi lattice $(\Pic_{X/k}/ \Pic_{X/k}^{0}) (k)$. Here, we denote the Picard functor of $X$ over $k$ and its identity component by $\Pic_{X/k}$ and $\Pic_{X/k}^{0}$.  \footnote{Note that $\Pic_{X/k}(k)$ is not necessarily equal to $\Pic(X)$, but $\Pic_{X/k}(k)\otimes_{\Z}\Q = \Pic (X) \otimes_{\Z}\Q$ holds in general.}
\begin{enumerate}
\item 
The \emph{positive cone} $\mathcal{C}_{X} \subset \Lambda_{X,\R}$ is the connected component of the cone of all elements $\lambda \in \Lambda_{X,\R}$ with $(\lambda,\lambda)>0$ satisfying that $\mathcal{C}_{X}$ contains the ample divisor classes. 
Here, $(\ast, \ast)$ is the intersection number (resp.\ Beauville--Bogomolov form).
We denote the closure of $\pos_{X}$ in $\Lambda_{X,\R}$ by $\overline{\mathcal{C}_{X}}$. We denote the convex hull of $\overline{\pos_{X}} \cap \Lambda_{X,\Q}$ by $\pos_{X}^{+}$.
\item
The \emph{ample cone} $\amp_{X} \subset \mathcal{C}_{X}$ is the cone generated by all ample divisor classes. We denote the closure of $\amp_{X}$ in $\Lambda_{X,\R}$ by $\nef_{X}.$
We denote the convex hull of $\nef_{X} \cap \Lambda_{X,\Q}$ in $\Lambda_{X,\R}$ by $\nef_{X}^{+}$.
\item
The \emph{movable cone} $\mv_{X} \subset \overline{\mathcal{C}}$ is the cone generated by movable divisor classes. 
Here, we say a divisor $D$ of $X$ is movable if the base locus of the linear system of $D$ has codimension $\geq 2$.
We denote the closure of $\mv_{X}$ in $\Lambda_{X,\R}$ by $\overline{\mv}_{X}$.
We denote the convex hull of $\overline{\mv}_{X} \cap \Lambda_{X,\Q}$ in $\Lambda_{X,\R}$ by $\mv_{X}^{+}$.
\item 
A \emph{polarization} (resp.\,\emph{quasi-polarization}) on $X$ is an element $L \in \Pic_{X/k}(k)$ which is ample (resp.\,nef and big).  
\end{enumerate}
\end{definition}

\begin{lemma}\label{conerat}
Let $k$ be a field, and $k_{s}$ the separable closure of $k$. 
Let $X$ be a smooth projective variety over $k$.
Then the following hold.
\begin{enumerate}
\item
$\amp_{X} = \amp_{X_{k_{s}}} \cap \Lambda_{X,\R}$.
\item
$\nef_{X} = \nef_{X_{k_{s}}} \cap \Lambda_{X,\R}$.
\item
$\mv_{X} = \mv_{X_{k_{s}}} \cap \Lambda_{X,\R}$.
\item
$\overline{\mv}_{X} = \overline{\mv}_{X_{k_{s}}} \cap \Lambda_{X,\R}$.
\end{enumerate}
\end{lemma}

\begin{proof}
First, we prove the assertions (1) and (3).
Take an element $x= \sum a_{i} x_{i}$ of the right-hand side of (1) (resp.\,(3)), where $a_{i}$ is a positive real number and $x_{i} \in \Lambda_{X_{k_{s}}}$  is an ample divisor class (resp.\,movable divisor class). Take a finite Galois extension $k'/k$ such that $x_{i} \in \Lambda_{X_{k'}}$ for any $i$.
We put $x_{i}' =  \sum_{\sigma \in \Gal (k'/k)} \sigma (x_{i})/ \# (\Gal (k'/k))$, which is also an ample divisor class (resp.\,movable divisor class).
Then we have $x = \sum a_{i} x_{i}'$, which is an element of the left-hand side.
The assertions (2) and (4) follow from \cite[Lemma 3.8]{Bright2019}.
\end{proof}

Finally, we introduce a notation for the set of twists.
\begin{definition}
Let $k$ be a field, and $X$ a variety over $k$.
Let $k'/k$ be an extension of fields.
We put 
\begin{align*}
\Tw_{k'/k}(X) \coloneqq
\left\{
Y \colon \textup{variety over }k \left.\mid Y_{k'} \simeq_{k'} X_{k'}
\right. 
\right\}/k\textup{-isom}.
\end{align*}
\end{definition}

\subsection{Almost abelian grouops}
In this section, we prove the finiteness of the group cohomology of an almost abelian group, which will be used to prove Theorem \ref{hypkfintwist} for an \hyp variety of geometric Picard rank $2$.
Moreover, combining with Oguiso's result, we also propose the finiteness of twists for certain Calabi--Yau threefolds.

First, we recall the definition of almost abelian groups, which was defined by Oguiso.
\begin{definition}[{cf.\,\cite[Section 8]{Oguiso2008}}]
\label{almab}
A group $G$ is \emph{an almost abelian group of finite rank $r$} if there exists a normal subgroup $G^{0}$ of $G$ of finite index, such that there exists a finite normal subgroup $K$ of $G^{0}$ with the exact sequence
\begin{equation}\label{ex seq}
1 \rightarrow K \rightarrow G^{0} \rightarrow \Z^{r} \rightarrow 1.
\end{equation}
\end{definition}

\begin{prop}
\label{propgalcoh}
Let $\Gamma$ be a finite group.
Let $G$ be an almost abelian group of finite rank $r$ admitting $\Gamma$-action $\sigma \colon \Gamma \rightarrow \Aut (G)$.
Then the group cohomology $H^{1}(\Gamma, G)$ is finite.
\end{prop}
\begin{proof}
Let $G^{0}$ be as in Definition \ref{almab}.
We put
\[
G^{1} \coloneqq \bigcap_{\gamma \in \Gamma} \gamma G^{0}.
\]
We note that $G^{1}$ is a normal subgroup of $G$ of finite index, since
\[
[G^{0} \colon G^{0}\cap\gamma(G^{0})] \leq [G\colon \gamma(G^{0})].
\]
Then the exact sequence (\ref{ex seq}) induces the exact sequence
\[
1 \rightarrow K\cap G^{1} \rightarrow G^{1} \rightarrow \Z^{r}.
\]
Since $G^{1}$ is of finite index in $G^{0}$, the image of the right arrow is also of finite index, and so it is isomorphic to $\Z^{r}$.
Therefore, replacing $G^{0}$ with $G^{1}$, we may assume that $G^{0}$ is stable under the $\Gamma$-action.
Then $K$ is also stable under the $\Gamma$-action.
Therefore, by \cite[Proof of D.1.7]{Degtyarev}, $H^1 (\Gamma, G)$ is finite.
\end{proof}

\begin{remark}
\label{remarkgalcoh}
By the proof of Proposition \ref{propgalcoh}, the number $\# H^{1} (\Gamma, G)$ is bounded above by the constant which depends only on the order $\# \Gamma$, the index of $G^{0}$ in $G$, the order $\# K$, and $r$ in Definition \ref{almab}.
Here, we note that the number of all the isomorphism classes of $\Gamma$-module $\Z^{r}$ is bounded above by $r$ and $\#\Gamma$ by \cite[Section 5, (a)]{Borel1963}, and thus so is $\# H^{1} (\Gamma, \Z^{r})$.
\end{remark}

\begin{definition}
Let $k$ be a field.
A Calabi--Yau variety over $k$ is a smooth projective variety satisfying that $\omega_{X/k} \simeq \oo_{X}$ and $H^{1}(X, \oo_{X})=0$.
\end{definition}

\begin{corollary}
Let $k'/k$ be a finite extension of fields of characteristic $0$, $X$ a variety over $k$.
Then the set $\Tw_{k'/k} (X)$ is finite in the following cases.
\begin{enumerate}
\item
$X$ is a Calabi--Yau threefolds satisfying that $\rk \Pic(\barX) \leq 3$.
\item
$X$ is a Calabi--Yau varieties of odd dimension satisfying that $\rk \Pic(\barX) \leq 2$.
\end{enumerate}
\end{corollary}
\begin{proof}
First, take a finitely generated subfield $k_{0} \subset k$ and a variety $X_{0}$ over $k_{0}$ such that $X_{0,k} \simeq X.$
Fix an embedding $\iota \colon \overline{k_{0}} \hookrightarrow \C$.
Since a Calabi--Yau variety $X$ has no infinitesimal automorphism by $h^{0,1} (X)= h^{n-1,0} (X) = 0$, the automorphism scheme $\Aut_{X_{0}/k_{0}}$ is unramified over $k_{0}.$
Thus we have 
\[
\Aut (\barX) \simeq \Aut (X_{0,\overline{k_{0}}}) \simeq \Aut (X_{0,\C}).
\]
Therefore, $\Aut (\barX)$ is finite or almost abelian of rank $1$ by \cite[Theorem 1.2]{Oguiso2014} and \cite[Theorem 1.1]{Lazic2018}.
Since $\Aut (X_{k'}) \subset \Aut (\barX)$, the group $\Aut (X_{k'})$ is also finite or almost abelian of rank $1$.
By Proposition \ref{propgalcoh}, the set $\Tw_{k'/k} (X)$ is a finite set.
\end{proof}

\section{The case of surfaces}
\label{Surfaces}
\subsection{K3 surfaces}
In this subsection, we discuss the finiteness of $\Tw_{k'/k}$ in the case of K3 surfaces.
First of all, we recall the definition of K3 surfaces.

\begin{definition}\label{defK3}
Let $k$ be a field.
A \emph{K3 surface over $k$} is a smooth projective surface $X$ over $k$
satisfying that $\omega_{X/k} \simeq \oo_{X}$ and $H^{1}(X, \oo_{X})=0$.
A K3 surface $X$ over $k$ is called \emph{supersingular} if the Picard rank of $X_{\overline{k}}$ is $22$.
\end{definition}

The main theorem of this section is the following.

\begin{theorem}\label{fintwistK3}
Let $k$ be a field, and $X$ a K3 surface over $k$.
Let $k'/k$ be a finite extension of fields.
Suppose that the characteristic of $k$ is different from $2$ or $X$ is not supersingular.
Then, the set
$\Tw_{k'/k}(X)$ 
is a finite set.
\end{theorem}

\begin{remark}
If $k'/k$ is a solvable extension, then Theorem \ref{fintwistK3} can be proved by the Galois cohomological argument as in the proof in \cite[Proposition 2.4.1]{Lieblich2014}.
\end{remark}

To prove Theorem \ref{fintwistK3}, we recall the cone conjecture for K3 surfaces.

\begin{theorem}[Cone conjecture]\label{coneK3}
Let $k$ be a field, and $X$ a K3 surface over $k$.
Suppose that characteristic of $k$ is not equal to $2$ or $X$ is not supersingular.
Then, the action of $\Aut(X)$ on $\nef^{+}_{X}$ admits a rational polyhedral fundamental domain.
\end{theorem}
\begin{proof}
See \cite[Corollary 3.15]{Bright2019} for the case of $\chara k \neq 2$.
They use the hypothesis on characteristic only when they cite the result over an algebraically closed field \cite{Lieblich2018} (see \cite[Proof of Proposition 3.10]{Bright2019}), which is valid for non-supersingular K3 surface in characteristic $2$ (see remarks at the end of \cite{Lieblich2018}).
\end{proof}

The cone conjecture implies the finiteness of polarizations of bounded degree, by the following easy lemma.

\begin{lemma}\label{finquasipol}
Let $\Lambda$ be a $\Z$-lattice (i.e. $\Z$-module with a symmetric bilinear pairing valued in $\Z$) of index $(1,\rho-1)$, $\pos_{\Lambda} \subset \Lambda_{\R}$ a positive cone (i.e.\, one connected component of the cone of all elements $\lambda \in \Lambda_{\R}$ with $(\lambda, \lambda) > 0$), $\overline{\pos_{\Lambda}}$ the closure of $\pos_{\Lambda}$
,and $\Pi \subset \overline{\pos_{\Lambda}}$ a rational polyhedron.  
We fix a positive integer $d \in \Z$.
Then there exist only finitely many integral elements of square $d$ in $\Pi$.
In particular, in the setting of Theorem \ref{coneK3}, the set of quasi-polarizations on $X$ of degree $d$ modulo $\Aut (X)$ is finite.
\end{lemma} 

\begin{proof}
This is a well-known argument (see \cite[Chapter8, Corollary 4.10]{Huybrechts2016} for the ample case), but for the sake  of completeness, we include it.
The second statement follows by applying $\Lambda = \Lambda_{X}$.
Therefore, we will prove the first statement.
First, we note that for any $y,z \in \overline{\pos_{\Lambda}}$, we have $(y,z) \geq 0$.
Let $x_{1}, \ldots, x_{n} \in \Pi \cap \Lambda$ be a system of integral generators of a rational polyhedron. 
We put $b_{i,j} \coloneqq (x_{i}, x_{j}),$ and $M \coloneqq \max b_{i,j}$.
Let $x = \sum_{i=1}^{n} a_{i} x_{i} \in \Pi \cap \Lambda$ be a square $d$ element.
We have 
\begin{align*}
(x,x_{i}) = \sum_{1 \leq j \leq n} a_{j}b_{i,j} >0
\end{align*}
for any $1\leq i \leq n$.
Indeed, if $(x,x_{i}) = 0$, then we have $(x_{i}, x_{i}) <0 $ by the assumption on index, and it contradicts $(x_{i},x_{i}) \geq 0$.
Since the left-hand side is a positive integer, there exists $1 \leq j_{i} \leq n$ such that
$a_{j_{i}} \geq n^{-1}M^{-1}$ and $b_{i, j_{i}}>0$ (so $b_{i,j_{i}} \geq 1$).
On the other hand, since 
\begin{align*}
(x, x) = \sum_{i,j=1,\ldots, n} a_{i} a_{j} b_{i,j}=d,
\end{align*}
we have $a_{i}a_{j_{i}}b_{i,j_{i}} \leq d$.
Therefore, we have $a_{i} n^{-1}M^{-1} \leq d$.
Therefore, such $x$ should be located on a bounded domain that is independent of $x$, so it finishes the proof.
\end{proof}

\begin{lemma}\label{weylK3}
Let $k$ be a field, and $k'$ a separable algebraic extension of $k$. 
Let $X$ be a K3 surface over $k$.
Then there exists a positive number $d$ such that 
for any $Y \in \Tw_{k'/k} (X),$ the variety $Y$ admits a degree $d$ polarization.
\end{lemma}
\begin{proof}
We may assume that $k'$ is the separable closure of $k$.
Note that $\Pic_{Y/k}$ is isomorphic to $\Pic_{X/k}$ after taking the base changes to $k'$.
Let $\la$ be the $\Z$-lattice $\Pic_{X_{\overline{k}}/\overline{k}}(\overline{k})$ = $\Pic_{X_{k'}/k'}(k')$ (see \cite[Lemma 3.1]{Ito2018}).
Then the set of conjugacy class of actions $\Gal (k'/k) \rightarrow \algO (\la)$ is finite by \cite[Section 5, (a)]{Borel1963}.
Take $Y_{1}, Y_{2} \in \Tw_{k'/k}(X)$ such that there exists a lattice isometry 
$\phi \colon \Pic_{Y_{1}/k}(k) \simeq \Pic_{Y_{2}/k}(k)$ which is induced by an isometry $\widetilde{\phi} \colon \Pic_{Y_{1}/k} (k') \simeq \Pic_{Y_{2}/k} (k')$.
It is enough to show that $Y_{1}$ and $Y_{2}$ admit a polarization of the same degree.
Without loss of generality, we may assume $\phi$ preserves a positive cone.
Take an ample class $y$ in $\Pic_{Y_{1}/k}(k)$,
By \cite[Proposition 3.7]{Bright2019},
there exists $r \in R_{Y_{2}}$ such that 
$r \circ \phi (y) \in \nef_{Y_{2}}$. 
Here, the group $R_{Y_{2}}$ is the Galois fixed part of the Weyl group of $Y_{2, k'}$ (see \cite[Definition 3.3]{Bright2019} for precise definition).
We note that $y$ has nonzero intersection numbers with any $(-2)$-classes in $\Pic_{Y_{1}/k}(k')$ (i.e.\, $x \in \Pic_{Y_{1/k}} (k')$ such that $(x, x)= -2$).
Therefore, the nef class $r \circ \phi (y) \in \Pic_{Y_{2/k}}(k)$ has a non-zero intersection number with any $(-2)$-classes in $\Pic_{Y_{2}/k} (k')$ and thus it is also an ample class.
\end{proof}

\subsection*{Proof of Theorem \ref{fintwistK3}}
First, we note that we may assume $k'/k$ is a finite Galois extension.
Indeed, since an infinitesimal automorphism of a K3 surface is trivial, 
the automorphism scheme $\Aut_{X/k}$ is an unramified scheme over $k$.
Therefore, for any K3 surface $X,Y$ over $k$,
the isomorphism scheme $\Isom_{X,Y/k}$ is an unramified scheme over $k$, since $\Isom_{X_{\overline{k}},Y_{\overline{k}}/\overline{k}}$ is empty or non-canonically isomorphic to $\Aut_{X_{\overline{k}}/\overline{k}}$. 
Therefore, for the separable closure $k''$ of $k$ in $k'$, we have
$\Tw_{k'/k} (X) = \Tw_{k''/k}(X)$.

In the following, we suppose that $k'/k$ is a finite Galois extension.
Let $d$ be an integer as in Lemma \ref{weylK3}.
For any $Y \in \Tw_{k'/k}(X),$ take a degree $d$ polarization $L_{Y}$.
By Lemma \ref{finquasipol}, we can take $M_{1}, \ldots, M_{m} \in \Pic_{X/k}(k')$ be a complete system of representatives of the set of polarizations of degree $d$ on $X_{k'}$ modulo $\Aut (X_{k'})$.
We put
\[
T_{i} :=
\left\{
(Y,L) \left| 
\begin{array}{l}
Y\colon \textup{K3 surface over } k, \\
L\colon \textup{polarization on } Y, \\
(Y,L)_{k'} \simeq (X_{k'}, M_{i})  
\end{array}
\right.
\right\}/k\textup{-isom}.
\]
Then $Y \mapsto (Y,L_{Y})$ gives an inclusion of sets
$\Tw_{k'/k} (X) \hookrightarrow \bigsqcup_{1 \leq i \leq m}{T_{i}}$. 
Therefore, it suffices to show the finiteness of $T_{i}$.
We may assume that $T_{i}$ is non-empty and fix an element $(Y_{0}, L_{0}) \in T_{i}$.
Then $T_{i}$ is isomorphic to $H^1 (\Gal(k'/k), \Aut (Y_{0, k'}, L_{0, k'}))$, which is finite by the finiteness of automorphism of 
polarized K3 surfaces (e.g.\, see \cite[Chapter 5, Proposition 3.3]{Huybrechts2016}). 

\subsection{Enriques surfaces}
First, we recall the definition of an Enriques surface.

\begin{definition}
Let $k$ be a field. An \emph{Enriques surface over $k$} is a smooth projective surface over $k$ satisfying $\omega_{X_{\overline{k}}/\overline{k}} \equiv \oo_{X_{\overline{k}}}$ and $b_{2}(X_{\overline{k}})=10$, where $\equiv$ denotes the numerical equivalence.
\end{definition}

In this section, we prove the finiteness of twists for a finite extension for Enriques surfaces.
If the base field has characteristic $\neq 2$, the problem is reduced to the case of K3 surfaces since an automorphism group of a K3 surface has at most finitely many conjugacy classes of Enriques involution (see \cite[Lemma 3.6]{Takamatsu2020b}).
But if the base field has characteristic $2$, K3 double cover may not exist in general.
To avoid this difficulty, we use the following lifting result given by Liedtke.

\begin{theorem}\label{liftEnr}
Let $k$ be a field of characteristic $2$.
Let $X$ be an Enriques surface over $k$.
There exists a finite separable 
extension $k'/k$ and a complete discrete valuation ring $R$ of mixed characteristic $(0,2)$ with residue field $k'$ such that
there exists a smooth projective scheme $\mathcal{X} \rightarrow \Spec R$ such that $\mathcal{X}_{k'} \simeq X_{k'}$.
\end{theorem}

\begin{proof}
If $k=\overline{k}$, this theorem follows from \cite[Theorem 4.10]{Liedtke2015}.
By \cite[Proposition 3.4]{Liedtke2015}, 
replacing $k$ with a finite separable extension if necessary, we may assume that 
$X$ admits 
a line bundle $L$ such that $L_{\overline{k}} \in \Pic ( X_{\overline{k}})$ gives a Cossec--Verra polarization in the sense of \cite[Definition 3.2]{Liedtke2015}.
As in the argument in \cite[Theorem 3.1]{Liedtke2015}, $L$ defines a birational morphism $\nu \colon X \rightarrow X'$, where $X'_{\overline{k}}$ is an Enriques surface with at worst Du Val singularities, and $X'$ admits an ample Cossec--Verra polarization $\nu_{\ast} (L)$.
By Artin's simultaneous resolution, it is enough to construct a lifting of $X'$.
Take a complete discrete valuation ring $R$ of mixed characteristic $(0,2)$ with residue field $k$ satisfying that $\sqrt 2 \in R$ is a uniformizer.
We note that $\Pic^{\tau}_{X'/k}$ is written as $G_{b,a}$ with $a,b \in k$ with $ab = 2$ by the Oort--Tate classification.
As in the proof of \cite[Theorem 4.9]{Liedtke2015}, we can find lifts $a', b' \in R$ with $a' b' =2$.
Therefore, we have a lift of $\Pic^{\tau}_{X'/k}$ to $R$.
By \cite[Theorem 4.7]{Liedtke2015}, the variety $X'$ admits a formal lift to $R$, which is algebraizable 
by \cite[Lemma2.4]{Takamatsu2020b} (see also \cite[Proposition 4.4]{Liedtke2015}).
\end{proof}

\begin{remark}
In Theorem \ref{liftEnr}, the generic fiber $\mathcal{X}_{\eta}\coloneqq \mathcal{X}_{\Frac(R)}$ is automatically an Enriques surface.
Indeed, the second Betti number of $\mathcal{X}_{\overline{\eta}}$ is 10 by the proper and smooth base change theorems.
Moreover, for any curve $C \subset \mathcal{X}_{\eta}$, we have $K_{\mathcal{X}_{\eta}/\eta}\cdot C = K_{X_{k'}/k'} \cdot C_{\mathrm{s}}= 0$, where $C_{\mathrm{s}}$ is the special fiber of the closure of $C \subset \mathcal{X}$.
\end{remark}

\begin{lemma}
\label{finpolEnr}
Let $k$ be a field of characteristic $2$.
Let $X$ be an Enriques surface over $k$.
Then there exists a finite extension $k'$ of $k$ satisfying that for any extension $k''$ of $k'$ and any positive integer $d$,
the set of polarizations of degree $d$ on $X_{k''}$ modulo $\Aut(X_{k''})$ is finite.
\end{lemma}

\begin{proof}
Take a finite field extension $k'/k$, a complete discrete valuation ring $R$, and a characteristic $0$ lift $\mathcal{X}$ over $R$ as in Theorem \ref{liftEnr}.
Let $F$ be the fractional field of $R$.
By \cite[Lemma 2.4]{Takamatsu2020b}, for any element $L \in \Pic_{X/k} (k)$, $L^{\otimes 2}$ lifts to a line bundle in $\Pic_{\mathcal{X}/R}(R)$.
We can take a finite extension $F'/F$ such that $\Aut(\mathcal{X}_{F'}) = \Aut(\mathcal{X}_{\overline{F}})$ and $\Pic(\mathcal{X}_{F'}) = \Pic(\mathcal{X}_{\overline{F}})$ (we note that $\Aut (X_{\overline{F}})$ is finitely generated by \cite[Theorem 3.3]{Dolgachev1984}).
In the following, we replace $k'$ with the residue field of the normalization $R'$ of $R$ in $F'$.
Then for any extension $k''/k'$, there exists a complete discrete valuation ring $R''$ with residue field $k''$ satisfying that $R' \subset R''$.
Let $F''$ be the fraction field of $R''$.
Now by Lemma \ref{finquasipol} and the cone conjecture for Enriques surfaces of characteristic $0$ (\cite[Theorem 2.1]{Kawamata1997}), 
the set $O$ of $\Aut (\mathcal{X}_{F''})$-orbit in degree $4d$ ample line bundles on $\mathcal{X}_{F''}$ is a finite set.
We define the subset $O' \subset O$ to be the set of orbits in $O$ which contains an ample line bundle $L$ on $\mathcal{X}_{F''}$ satisfying $\spp (L)$ is ample.
Here, 
\[
\spp\colon \Pic (\mathcal{X}_{F''}) \rightarrow \Pic (X_{k''})
\]
is the natural specialization map.
Take a complete system of representatives $M_{1}, \dots, M_{n}$ of $O'$ such that $\spp (M_{i})$ is ample for $1\leq i \leq n$.
Note that the set of elements in $\Pic_{X_{k''}/k''}(k'')$ two times of which appear in specializations of $M_{1}, \ldots, M_{n}$ is a finite set. 
We put this set as 
$\{
L_{1}, \ldots, L_{m}
\}$.
We will prove that this set represents all the elements of 
polarizations of degree $d$ on $X_{k''}$ modulo $\Aut (X_{k''})$.
Take a polarization $L$ of degree $d$ on $X_{k''}$.
Take a lift $M \in \Pic (\mathcal{X}_{R''}) \simeq \Pic (\mathcal{X}_{F''})$ of $L^{\otimes 2}$.
Then $M$ is a polarization of degree $4d$ on $\mathcal{X}_{F''}$.
Therefore, there exists an automorphism $\gamma \in \Aut (\mathcal{X}_{F''})$ and $1 \leq i \leq n$ such that $\gamma (M) = M_{i}$.
Taking the specialization, we have $\overline{\gamma} (L^{\otimes 2}) = \spp (M_{i})$.
Here, $\overline{\gamma} \in \Aut (X_{k''})$ is a specialization of $\gamma$ which exists by \cite[Corollary 1]{Matsusaka1964}.
We note that the specialization of automorphism is compatible with $\spp$ by the argument in the proof of \cite[Theorem 2.1]{Lieblich2018}.
Therefore, $\overline{\gamma} (L)$ is one of an element of $\{L_{1},\ldots, L_{m}\}$.
\end{proof}

\begin{theorem}
\label{fintwistEnr}
Let $k$ be a perfect field of characteristic $2$, and $X$ an Enriques surface over $k$.
Let $k'/k$ be a finite extension of fields.
Then the set 
$\Tw_{k'/k}(X)$ is finite.
\end{theorem}

\begin{proof}
We may assume that $k'/k$ is a Galois extension. 
As in the argument in Lemma \ref{weylK3}, the set of lattice isometry classes of $\Lambda_{Y}$ for $Y\in \Tw_{k'/k} (X)$ is a finite set. 
We denote this finite set by $S$.
We fix $Y \in \Tw_{k'/k}(X)$.
First, by \cite[Chapter 8, Remark 2.2]{Huybrechts2016}, there exists a polarization $L_{Y}$ satisfying the following:
For any $(-2)$-class $x$ (i.e.\,$x \in \Pic(Y_{\overline{k}})$ with $(x,x)= -2$) satisfying $x^{\perp} \nsupseteq \Lambda_{Y}$,  we have $(L,x) \neq 0$.
Take a $Y' \in \Tw_{k'/k}(k)$ with an isometry $\phi \colon \Lambda_{Y}\simeq \Lambda_{Y'}$.
Let $R_{Y'_{\overline{k}}} \subset \algO (\Lambda_{Y'_{\overline{k}}})$ be the subgroup generated by reflections with respect to $(-2)$-curves.
Then $\nef_{Y'_{\overline{k}}}$ is a fundamental domain with respect to the action of $R_{Y'_{\overline{k}}}$ on the positive cone $\pos_{Y'_{\overline{k}}}$ by \cite[Proposition 2.2.1]{Cossec2021}.
We put $R_{Y'} \coloneqq R_{Y'_{\overline{k}}}^{\Gal (\overline{k}/k)} \subset \algO (\Lambda_{Y'})$. 
Now we can show that $\nef_{Y'}$ is a fundamental domain with respect to the action of $R_{Y'}$ by the same argument as in \cite[Proposition 3.7]{Bright2019} (see also the proof of Proposition \ref{fundweylrational}).
Replacing $\phi$, we may assume that $\phi (L_{Y}) \in \pos_{Y'}$. 
Therefore, there exists $r\in R_{Y'}$ such that $r \circ \phi (L_{Y})$ is a quasi-polariation of $Y'$.
Moreover, by the choice of $L$, the element $r \circ \phi (L_{Y})$ has positive intersection with any $(-2)$-curve on $Y'_{\overline{k}}$ (since there exists a polarization on $Y'$). 
Thus, by \cite[Proposition 2.1.5]{Cossec2021}, $r \circ \phi (L_{Y})$ is a polarization on $Y'$.
Therefore, there exists an integer $d$ such that $Y \in \Tw_{k'/k}(X)$ admits a polarization $M_{Y}$ of degree $d$ (we can take $d$ as the least common multiple of $(L_{Y},L_{Y})$ for $[Y] \in S$).

By Lemma \ref{finpolEnr}, we may assume that the set of polarizations of degree $d$ on $X_{k'}$ modulo $\Aut (X_{k'})$ is a finite set. 
We take a complete system of representatives $L_{1}, \ldots, L_{m} \in \Lambda_{X_{k'}}$ of this set.
Now we put
\[
T_{i} \coloneqq
\left\{
(Y,L) \left| 
\begin{array}{l}
Y\colon\textup{Enriques surface over } k, \\
L\colon\textup{polarization on } Y, \\
(Y,L)_{k'} \simeq (X_{k'}, L_{i})  
\end{array}
\right.
\right\}/k\textup{-isom}.
\]
Then $Y \mapsto (Y,M_{Y})$ gives an inclusion of sets $\Tw_{k'/k} (X) \rightarrow \bigsqcup_{1\leq i \leq m} T_{i}$.
On the other hand, each $T_{i}$ is a finite set by \cite[Proposition 2.26]{Brion2018} and \cite[Theorem 3]{Dolgachev2016}. Therefore, it finishes the proof.
\end{proof}

\section{Cone conjecture for \hyp varieties}
\label{Hyperkahlervarieties}
In this section, we prove the cone conjecture for \hyp varieties over a field of characteristic $0$, that are not necessarily algebraically closed.

\begin{definition}\label{Pex}
Let $k$ be an algebraically closed field of characteristic $0$.
Let $X$ be an \hyp variety over $k$.
\begin{enumerate}
\item
A prime divisor $E \subset X$ is \emph{exceptional} if $(E,E) < 0$.
We denote the set of prime exceptional divisors by $\Pex(X)$.
\item
By the same argument as in \cite[Proposition 6.2]{Markman2011}, a prime exceptional divisor $E$ defines the reflection $r_{E} \in \algO (\Lambda_{X})$.
We denote the subgroup of $\algO (\Lambda_{X})$ generated by the reflections with respect to prime exceptional divisors by $W_{Exc}$. 
Note that, if $k= \C$, Markman defines $W_{Exc}$ as the subgroup of $\algO(H^2 (X,\Z))$ generated by the reflections with respect to prime exceptional divisors (cf.\ \cite[p.~263]{Markman2011}).
However, by \cite[Theorem 6.18 (5) and Lemma 6.23 (2)]{Markman2011}, 
the subgroup of $\algO(H^2(X,\Z))$ generated by reflections by prime exceptional divisors is isomorphic to its image in $\algO(\Lambda_X)$, i.e. the subgroup of $\algO(\Lambda_X)$ generated by reflections by prime exceptional divisors. 
\end{enumerate}
\end{definition}

\begin{prop}\label{fundweyl}
Let $k$ be an algebraically closed field of characteristic $0$.
Let $X$ be an \hyp variety over $k$.
Then $\overline{\mv}_{X} \cap \pos_{X}$ is a fundamental domain for the action of $W_{Exc}$ on $\pos_{X}$, cut out by closed half-space associated to elements of $\Pex (X)$.
\end{prop}
\begin{proof}
See \cite[Lemma 6.22]{Markman2011}.
\end{proof}

\begin{definition}
Let $k$ be a field of characteristic $0$. 
Let $X$ be an \hyp variety over $k$.
As in Definition \ref{Pex}, we have the subgroup 
$W_{Exc} \subset \algO (\Lambda_{\barX})$.
Define $R_{X}$ to be the fixed part $W_{Exc}^{G_{k}}$, where $G_{k}$ denotes the absolute Galois group of $k$.
We note that $R_{X}$ acts on $\Lambda_{X}$.
\end{definition}

\begin{prop}\label{fundweylrational}
Let $k$ be a field of characteristic $0$.
Let $X$ be an \hyp variety over $k$.
Then $\overline{\mv}_{X} \cap \pos_{X}$ is a fundamental domain for the action of $R_{X}$ on $\pos_{X}$. 
\end{prop}
\begin{proof}
This follows from the same argument as in \cite[Proposition 3.7]{Bright2019}, but we include it.
First, we prove that any class $x$ of $\pos_{X}$ are $R_{X}$-equivalent to an element of $\overline{\mv}_{X} \cap \pos_{X}$. 
If we suppose that $x$ has trivial stabilizer in $R_{\barX} = W_{Exc}$, then there exists a unique $g \in W_{Exc}$ such that $gx \in \overline{\mv}_{\barX} \cap \pos_{\barX}$.
For any $\sigma \in G_{k}$, we have
\[
(^\sigma g)(x) = \sigma g (\sigma^{-1} (x)) = \sigma (g (x)) \in \overline{\mv}_{\barX} \cap \pos_{\barX},
\]
since the Galois action preserves $\mv_{\barX}$.
By the assumption, we have $^{\sigma}g = g$, and thus $g \in R_{X}$.
Next, we consider the case where $x\in \pos_{X}$ has a non-trivial stabilizer.
By Proposition \ref{fundweyl}, $x$ lies in the following walls
\[
\bigcup_{E \in \Pex(\barX)} (E^{\perp} \cap \pos_{\barX}) \subset \pos_{\barX}.
\]
By the argument in \cite[Chapter 8, Remark 2.3]{Huybrechts2016},  this union is locally finite.
Therefore, one can take a sequence $(x_{n})$ of elements of $\pos_{X}$ such that $x_{n}$ converge to $x$ and all $x_{n}$ lying in the interior of the same chamber of $\pos_{X}$ with respect to the above walls.
Then there exists a unique $g \in R_{X}$ with $g(x_{n}) \in \overline{\mv}_{\barX} \cap \pos_{X}$, so is $g(x)$.

Now it suffices to prove that $\overline{\mv}_{\barX} \cap \pos_{X}$ intersects the translate by any non-trivial element of $R_{X}$ only in its boundary.
This follows from Proposition \ref{fundweyl} and \cite[Lemma 3.8]{Bright2019}.
\end{proof}

\subsection{Birational cone conjecture}

\begin{prop}
\label{indexbir}
Let $k$ be a field of characteristic $0$.
Let $X$ be an \hyp variety over $k$.
Let
\[
\rho \colon \Bir (X) \ltimes R_{X} \rightarrow \algO (\Lambda_{X}) 
\]
be the natural action.
Then the image of $\rho$ is a finite index subgroup, and the kernel of $\rho$ is a finite subgroup of $\Aut (X) \subset \Bir (X)$.
\end{prop}

\begin{proof}
First, we note that a birational automorphism that sends a very ample class to a very ample class is an automorphism and an automorphism group of a quasi-polarized \hyp variety is a finite group (\cite[Proposition 2.26]{Brion2018}).
Therefore, the case of $k=\C$ follows from \cite[Proposition 6.18, Lemma 6.23]{Markman2011}.
Since an \hyp variety have no infinitesimal birational automorphisms (\cite[Theorem 3.8]{Blanc2017}) or infinitesimal line bundles, the case of $k= \overline{k}$ follows.
Indeed, one can take a finitely generated subfield $k'$ of $k$ such that $X$ comes from $X'$ over $k'$.
Take an embedding $\overline{k'} \hookrightarrow \C$. 
Then $\Bir (X) = \Bir (X'_{\overline{k'}}) = \Bir (X'_{\C})$ by the above fact.
On the other hand, we have $\Lambda_{X} = \Lambda_{X'_{\overline{k'}}} = \Lambda_{\C}$ by the above fact.
Now we have $R_{X} = R_{X'_{\overline{k'}}} = R_{X'_{\C}}$ and we can reduce to the case of $k= \C$.

In the following, we treat the general case.
We know that
\[
\rho' \colon \Bir (X_{\overline{k}}) \ltimes R_{X_{\overline{k}}} \rightarrow \algO (\Lambda_{X_{\overline{k}}})
\]
has finite kernel and the image of finite index.
We note that the birational automorphism group satisfies a Galois descent property 
$\Bir (X_{\overline{k}})^{G_{k}} = \Bir (X)$. 
Since $\Bir (X_{\overline{k}})$ is a finitely generated group (\cite[Theorem 2]{Boissiere2012}), the action of $G_{k}$ on $\Bir (X_{\overline{k}})$ factors through a finite quotient.
Similarly, since $\Lambda_{X_{\overline{k}}}$ is a finitely generated group, the actions of $G_{k}$ on $R_{X_{\overline{k}}}$ and $\algO(\Lambda_{X_{\overline{k}}})$ factors through a finite quotient of $G_k$.
Therefore, by \cite[Lemma 3.12]{Bright2019},
the restriction
\[
\Bir (X) \ltimes R_{X} \rightarrow \algO (\Lambda_{X_{\overline{k}}})^{G_{k}}
\]
has finite kernel and the image of finite index.
By \cite[Proposition 2.2 (2)]{Bright2019},
$\algO(\Lambda_{X_{\overline{k}}})^{G_{k}}$ is of finite index in $\algO(\Lambda_{X_{\overline{k}}}, \Lambda_{X})$.
Here, the group
$\algO(\Lambda_{X_{\overline{k}}}, \Lambda_{X})$ is subgroup of $\algO(\Lambda_{\barX})$ stabilizing $\Lambda_{X}$ as a subset.
Moreover, by \cite[Proposition 2.2 (1), (2)]{Bright2019},
the map $\algO(\Lambda_{\barX}, \Lambda_{X}) \rightarrow \algO(\Lambda_{X})$
has finite kernel and the image of finite index.
Therefore the morphism
\[
\Bir (X) \ltimes R_{X} \rightarrow \algO (\Lambda_{X})
\]
has finite kernel and the image of finite index.
\end{proof}

We will study the structure of $R_{X}$ by following the method of \cite{Bright2019}.
\begin{definition}
\begin{enumerate}
\item
Let $W$ be a group, and $T \subset W$ a set of elements of order $2$ which generates $W$ as a group.
For $t_{i},t_{j} \in T$, we denote the order of $t_{i}t_{j}$ by $n_{i,j}$ (if the order is infinite, we set $n_{i,j}=0$).
We say $(W,T)$ is a Coxeter system if $W$ has the following representation.
\[
\langle t \in T \mid t^{2}=1, (t_{i}t_{j})^{n_{i,j}}=1 \textup{ if } n_{i,j}\neq 0 \rangle.
\]
\item
Let $(W,T)$ be a Coxeter system.
Let $G$ be a graph with vertices $T$ and such that $t_{i}, t_{j}$ are adjacent in $G$ if and only if $t_{i}$ does not commute with $t_{j}$.
We label an integer to each edge between $t_{i}$ and $t_{j}$ by $n_{i,j}-2$ for $n_{i,j}>0$ and $0$ otherwise. We say this graph $G$ is the Coxeter--Dynkin diagram of the Coxeter system $(W,T)$.
\item
Let $(W,T)$ be a Coxeter system.
The length $\ell (w)$ of $w\in W$ is the length of a shortest word consists of elements of $T$ that represents $w$. If $W$ is finite, there exists a unique $w_{0}\in W$ such that $\ell (w_{0})$ is greater than $\ell (w)$ for any $w \neq w_{0}$.
We say $w_{0}$ is the longest element of $(W,T)$.
\end{enumerate}
\end{definition}
The following is an analogue of \cite[Proposition 3.6, Remark 3.9]{Bright2019}.

\begin{prop}\label{lembircone}
Let $k$ be a field of characteristic $0$.
Let $X$ be an \hyp variety over $k$.
Let $I$ be a Galois orbit of prime exceptional divisors on $\barX$.
We denote the subgroup of $W_{Exc}$ generated by the reflections with respect to $E\in I$ by $W_{I}$.
Then the following holds.
\begin{enumerate}
\item
If $W_{I}$ is finite, then one of the following holds. 
\begin{enumerate}
\item
For any $E_{1} , E_{2} \in I$ with $E_{1} \neq E_{2}$, we have $(E_{1}, E_{2})=0.$
\item
For any $E_{1} \in I$, there exists a unique $E_{2} \in I$ such that $E_{1} \neq E_{2}$ and $(E_{1}, E_{2}) \neq 0$.
Moreover, in this case, we have $(E_{1}, E_{1}) = -2 (E_{1}, E_{2})$.
\end{enumerate}
We denote the set of Galois orbits $I$ of prime exceptional divisors 
such that $W_{I}$ is finite by $F$.
\item
For $I \in F$, we denote the longest element of the Coxeter system $(W_{I}, I)$ by $r_{I}$.
Then $(R_{X}, \{r_{I}| I \in F\})$ is a Coxeter system.
Moreover, $r_{I}$ is given by the reflection with respect to the sum $C_{I}$ of all classes in $I$.
\item
A class $\lambda \in \pos_{X}$ lies in $\overline{\mv}_{X}$ if and only if $(\lambda,E) \geq 0$ for any $E\in I$ and $I \in F$.
\end{enumerate}
\end{prop}
\begin{proof}
First, we note that $(W_{Exc}, \Pex)$ and $(W_{I}, I)$ are Coxeter systems (see \cite[Section 5.4]{Heckman2018} and \cite[Proposition 2.4.1 (i)]{Bjoerner2005}).

(1) Let $I$ be a Galois orbit of prime exceptional divisors such that $W_{I}$ is finite.
Take $E_{1}, E_{2} \in I$ with $E_{1} \neq E_{2}$. 
Then by \cite[Theorem 5.8]{Markman2011}, their classes in $\Lambda_{\barX}$ are different, so their classes are linearly independent.
We put $\alpha \coloneqq (E_{1}, E_{2})$ and $\beta \coloneqq (E_{1}, E_{1}) = (E_{2} , E_{2})$.
Note that $r_{E_{1}}$ and $r_{E_{2}}$ stabilize the $\R$-vector space
$W \coloneqq \R E_{1} \oplus \R E_{2},$ and their representation matrix is the following.
\[
r_{E_{1}}|_{W} = \left(
\begin{array}{cc}
\displaystyle -1 & \displaystyle -2 \frac{\alpha}{\beta} \\
& \\
\displaystyle 0 & \displaystyle 1
\end{array}
\right),
r_{E_{2}}|_{W} = \left(
\begin{array}{cc}
\displaystyle 1 & \displaystyle 0 \\
& \\
\displaystyle -2\frac{\alpha}{\beta} & \displaystyle -1
\end{array}
\right).
\]
Therefore, we have 
\[
r_{E_{1}} r_{E_{2}} |_{W}= \left(
\begin{array}{cc}
\displaystyle -1+ 4\frac{\alpha^{2}}{\beta^{2}} & \displaystyle 2\frac{\alpha}{\beta} \\
& \\
\displaystyle -2\frac{\alpha}{\beta} & \displaystyle -1
\end{array}
\right).
\]
Since $W_{I}$ is finite and the above matrix has $\Q$-coefficient, 
$r_{E_{1}}r_{E_{2}}|_{W}$ has the order $2,3,$ or $4$.
By seeing a trace, we can show that the order cannot be $4$.
Moreover, we have the order is $2$ if and only if $\alpha =0$,
and the order is $3$ if and only if $\beta = -2 \alpha$.
On the other hand, since the Coxeter--Dynkin diagram of $(W_{I}, I)$ is a finite union of finite trees (\cite[Exercise 1.4]{Bjoerner2005}), there exists a vertex of degree $\leq 1$.
Since $G_{k}$ acts transitively on the diagram, we have either every vertex have degree $0$ or every vertex have degree $1$.
Therefore, if $W_{I}$ is finite, then (a) or (b) holds.
Therefore, it finishes the proof of part (1).

(2) The first part follows from 
\cite[Theorem 1]{Geck2014}.
We will show the second part.
First, we consider the case (a). We write $I$ as $\{E_{1}, \ldots, E_{r} \}$.
Then $W_{I}$ is isomorphic to the Coxeter group $A_{1}^{r}$. Therefore, their longest element is given by $r_{E_{1}} r_{E_{2}} \cdots r_{E_{r}} = r_{C_{I}}$.
Next, we consider the case $(b)$. 
We write $I$ as $\{E_{1}, E_{1}', \ldots, E_{r}, E_{r}'\}$, where $(E_{i}, E_{i}') \neq 0.$
Then $W_{I}$ is isomorphic to the Coxeter group $A_{2}^{r}$, and their longest element is given by 
$(r_{E_{1}}r_{E_{1}'}r_{E_{1}}) \cdots (r_{E_{r}}r_{E_{r}'}r_{E_{r}}) 
= r_{E_{1}+E_{1}'} \cdots r_{E_{r} + E_{r}'} 
= r_{C_{I}}$.  Therefore, it finishes the proof of part (2).

(3) 
Let $F' \subset \Lambda_{\barX}$ be a set of all the classes $x \in \Lambda_{\barX}$ such that $(x,x) < 0$. 
By the argument in \cite[Remark 8.2.3]{Huybrechts2016}, 
the following union is locally finite.
\[
\bigcup_{x \in F'} (x^{\perp} \cap \pos_{\barX}) \subset \pos_{\barX}.
\]
Let $I$ be a Galois orbit of prime exceptional divisors such that $W_{I}$ is infinite.
Suppose that there exists a $\lambda \in \pos_{X}$ such that $(\lambda, E) = 0$ for some $E \in I$.
Then $\lambda$ is orthogonal to every element in $I$.
Therefore, $\lambda$ is fixed by the action of $W_{I}$, and therefore $\lambda$ is orthogonal to $w(E)$ for any $E \in I, w \in W_{I}$. 
Since $w r_{E} w^{-1} = r_{w(E)}$, by \cite[Theorem1]{Speyer2009}, the set $W_{I}E$ is infinite for some $E\in I$.
Therefore, $\lambda$ is orthogonal to infinitely many elements in $F'$, which is a contradiction.  
Thus, for any $E \in I,$ the wall $E^{\perp}$ does not meet with $\pos_{X}$.
Combining with Proposition \ref{fundweyl}, it finishes the proof.
\end{proof}
The main theorem in this subsection is the following.

\begin{theorem}\label{bircone}
Let $k$ be a field of characteristic $0$.
Let $X$ be an \hyp variety over $k$.
Then there exists a rational polyhedral cone $\Pi$ in $\mv_{X}^{+}$ which is a fundamental domain for the action of $\Gamma_{\Bir(X)}$ on $\mv_{X}^{+}$.
Here, $\Gamma_{\Bir(X)}$ means the image of the birational automorphism group via the natural action $\Bir(X) \rightarrow \algO (\Lambda_{X,\R})$.
\end{theorem}

\begin{proof}
Let $\Gamma$ be the image of $\rho \colon \Bir (X) \ltimes R_{X} \rightarrow \algO (\Lambda_{X})$, which is an arithmetic subgroup of $\algO(\Lambda_{X})$.
We choose an ample class $y \in \pos_{X} \cap \Lambda_{X}$.
The set 
\[
\Pi \coloneqq 
\{
x \in \pos_{X}^{+} \mid (\gamma x, y) \geq (x, y) \textup{ for all } \gamma \in \Gamma
\}
\]
is rational polyhedral as in the proof of \cite[Corollary 3.15]{Bright2019}.
Let $I$ be a Galois orbit of prime exceptional divisors such that $W_{I}$ is finite.
Let $C_{I}$ be as in Proposition \ref{lembircone}.
Taking $\gamma = r_{C_{I}}$, one can show that 
$x\in \Pi$ satisfies $(x, C_{I}) \geq 0$.
By Proposition \ref{lembircone} (3), $\Pi$ is contained in $\overline{\mv}_{X}.$
Now we will show that $\Pi$ satisfies desired properties.
For any $x \in \mv^{+}_{X}$, one can find $\phi \in \Bir (X)$ and $r \in R_{X}$ such that
$r\phi (x)$ lies in $\Pi \subset \overline{\mv}_{X}.$
On the other hand, $\phi (x) \in \overline{\mv}_{X}.$ 
Therefore, by Proposition \ref{fundweylrational}, we have $r \phi (x) = \phi (x) \in \Pi$.
It finishes the proof.
\end{proof}

\subsection{Automorphism cone conjecture}
\label{subsectautomcone}

In this subsection,
let $k$ be a field of characteristic $0$, $X$ an \hyp variety over $k$.
Let $\Sigma \subset \Lambda_{\barX}$ be the set 
\[
\Sigma \coloneqq
\left\{
f_{\ast} (e) \left| 
\begin{array}{l}
e \in \nef_{Y}^{\ast} \textup{ is integral, primitive, and extremal, and }\\
f\colon Y \dashrightarrow \barX \textup{ is a birational map of \hyp varieties over } \overline{k}
\end{array}
\right.
\right\}.
\]
Here, the cone $\nef_{Y}^{\ast} \subset \Lambda_{Y,\R}$ is the dual cone of $\nef_{Y}$.
We note that any element in $\Sigma$ is a primitive MBM class in the sense of \cite[Definition 2.14]{Amerik2017}.
\begin{definition}
\label{defbiramp}
Let $\mv^{0}_{\barX}$ be the interior of the movable cone.
We define the birational ample cone of $\barX$ by 
\[
\ba_{\barX} \coloneqq \mv_{\barX}^{0} \setminus (\cup _{\lambda \in \Sigma} \lambda^{\perp} \cap \mv_{\barX}^{0}).
\]
We denote the intersection $\ba_{\barX} \cap \Lambda_{X,\R}$ by $\ba_{X}$.
\end{definition}

\begin{prop}\label{biramp}
Let $f \colon Y\dashrightarrow X$ be a birational map of \hyp varieties over $k$.
Then the image $f_{\ast}(\amp_{Y})$ is a connected component of $\ba_{X}$. Moreover, every connected component of $\ba_{X}$ is of this form.
\end{prop}

\begin{proof}
The case of $k= \overline{k}$ follows from \cite[Proposition 2.1]{Markman2015}.
We treat the general case.
We note that
\begin{equation}\label{intersect}
f_{\overline{k},\ast}(\amp_{Y_{\overline{k}}}) \cap \Lambda_{X} = f_{\ast}(\amp_{Y}).
\end{equation}
Indeed, an element $f_{\overline{k},\ast} (x)$ of the left-hand side is Galois invariant, and since $f$ is defined over $k$, $x \in \amp_{Y_{\overline{k}}}$ is also Galois invariant. Therefore, Lemma \ref{conerat} gives desired equality.
Therefore, the first part of Proposition \ref{biramp} follows from the case of $k= \overline{k},$ since $f_{\ast}(\amp_{Y})$ is connected.

We show the remaining part of Proposition \ref{biramp}.
Let $A$ be a connected component of $\ba_{X}.$
From the case of $k=\overline{k}$, there exists a birational map $f \colon Z \dashrightarrow \barX$  of \hyp varieties over $\overline{k}$, such that $A$ is contained in $f_{\ast} (\amp_{Z}) \subset \ba_{\barX}.$
For any $\sigma \in G_{k}$, by the $k$-structure of $X$, we have an isomorphism 
\[
\phi_{\sigma} \colon \barX \rightarrow  \sigma(\barX).
\]
On the other hand, we have a birational map $^{\sigma}f\colon \sigma (Z) \dashrightarrow \sigma (\barX)$.
We want to show that $\psi_{\sigma}= ^{\sigma}f^{-1} \circ \phi_{\sigma} \circ f$ is an actual isomorphism, so that $\psi_{\sigma}$ gives $k$-structure on $Z$.
Take an element $L \in A$, and let $M \in \amp_{Z}$ such that $f_{\ast} (M) =L$.
Note that there exist classes $L'= \sigma (L)$ on $\sigma(\barX)$, $M' = \sigma (M)$ on $\sigma (Z)$.
Then we have
\[
\psi_{\sigma,\ast} (M) = ^{\sigma}f^{-1}_{\ast} \circ \phi_{\sigma,\ast} (L) = ^{\sigma}f^{-1} (L') = M'.
\]
Here, the second equality follows from $L \in \Lambda_{X,\R}.$
Therefore, $\psi_{\sigma}$ pulls back an ample class to an ample class, thus $\psi_{\sigma}$ is an isomorphism
(indeed, in this case, $\psi_{\sigma}$ pulls back a very ample class to a very ample class).
Therefore, $Z$ descends to a variety $Y$ over $k$, and by the definition of the $k$-structure on $Y$, a birational map $f$ is defined over $k$.
Combining with the equality (\ref{intersect}), it finishes the proof.
\end{proof}

\begin{theorem}[{\cite[Theorem 5.3]{Amerik2017}, \cite[Corollary 1.4]{Amerik2016}}]
Suppose that $b_{2} (\barX) \geq 5$.
Then $\Sigma$ have bounded Beauville--Bogomolov square.
\label{automconebarX}
\end{theorem}

In the following, we suppose that $b_{2} (\barX) \geq 5$.
Let $\Pi \subset \mv_{X}^{+}$ be a rational polyhedral fundamental domain as in Theorem \ref{bircone}.
By Theorem \ref{automconebarX} and \cite[Proposition 3.4]{Markman2015}, the set 
\[
\{\lambda \in \Sigma \mid \lambda^{\perp} \cap \Pi \cap \pos_{X} \neq \emptyset\}
\]
is finite.
This set divides $\Pi$ into a finite union of closed rational polyhedral subcones. We denote these subcones by 
\[
\Pi_{j}, \ \ j\in J,
\]
where $J$ is a finite index set.

\begin{lemma}\label{lemnef}
Let $f \colon Y\dashrightarrow X$ be a birational map of \hyp varieties over $k$.
Let $g$ be an element of $\Bir(X)$, such that $g(\Pi_{j})$ intersects $f_{\ast} (\amp_{Y})$.
Then  $g(\Pi_{j}) = f_{\ast}(\nef_{Y}) \cap g (\Pi)$.
\end{lemma}
\begin{proof}
The proof is the same as in \cite[Lemma 2.3]{Markman2015}, 
but we include the proof.
By replacing $f$ by $g^{-1} \circ f$, we may assume that $g =1$.
We denote the interior of $\Pi_{j}$ (resp.\,$\Pi$) by $\Pi_{j}^{0}$ (resp.\,$\Pi^{0}$).
It suffices to prove the equality 
\[
\Pi_{j}^{0} = f_{\ast} (\amp_{Y}) \cap \Pi^{0}.
\]
The inclusion $\subset$ follows from Proposition \ref{biramp}.
Another inclusion also follows since the right-hand side is convex, so connected and disjoint from the hyperplane $\lambda^{\perp}$ for every $\lambda \in \Sigma$ by Proposition \ref{biramp}.
\end{proof}

Let $f \colon Y\dashrightarrow X$ be a birational map of \hyp varieties over $k$.
Let $J_{f} \subset J$ be the subset consisting of indices $j$ such that their exists $g_{j} \in \Bir (X)$ such that $g_{j} (\Pi_{j})$ is contained in $f_{\ast}(\nef_{Y})$.

\begin{lemma}\label{birJ}
Let $f_{i} \colon Y_{i} \dashrightarrow X$ ($i=1,2$) be a birational map of \hyp varieties over $k$. Then the following hold.
\begin{enumerate}
\item
$Y_{1}$ is isomorphic to $Y_{2}$ over $k$ if and only if $J_{f_{1}} = J_{f_{2}}$.
\item
If $J_{f_{1}} \cap J_{f_{2}}$ is nonempty, then we have $J_{f_{1}} = J_{f_{2}}$.
\end{enumerate}
\end{lemma}

\begin{proof}
The proof is the same as in \cite[Lemma 2.4]{Markman2015}, but we include the proof.
For (1), first suppose that there exists an isomorphism $\phi \colon Y_{1} \rightarrow Y_{2}$.
Then $\psi \coloneqq f_{2}\phi f_{1}^{-1}$ sends $f_{1}(\nef_{Y_{1}})$ to $f_{2}(\nef_{Y_{2}})$.
Thus for any $g \in \Bir (X)$ and $j\in J$, $g (\Pi_{j}) \subset f_{1}(\nef_{Y_{1}})$ if and only if $\psi g (\Pi_{j}) \subset f_{2} (\nef_{Y_{2}})$.
Thus we have $J_{f_{1}} = J_{f_{2}}$.
For the remaining part of (1) and (2), suppose that there exists an element $j \in J_{f_{1}} \cap J_{f_{2}}$.
Then there exist $h_{i} \in \Bir (X)$ such that $h_{i} (\Pi_{j}) \subset f_{i} (\nef_{Y_{i}})$.
Thus $f_{2}^{-1}h_{2}h_{1}^{-1}f_{1}$ maps some ample class to an ample class.
Therefore, $f_{2}^{-1}h_{2}h_{1}^{-1}f_{1}$ is an isomorphism, and it finishes the proof.
\end{proof}

For a birational map $f \colon Y \dashrightarrow X$ of \hyp varieties, we put $J_{Y} \coloneqq J_{f}$ which depends only on an isomorphism class of $Y$ by Lemma \ref{birJ},

\begin{lemma}
\label{lemautcoset}
For any $j\in J_{X}$, the set 
\[
\{
g \in \Bir(X) \mid g(\Pi_{j}) \subset \nef_{X}
\}
\]
is an $\Aut (X)$-coset, i.e.\ it is equal to $\Aut (X) g_{j}$ for some $g_{j} \in \Bir (X)$.
Moreover, $\nef_{X}^{+}$ is the union of $\Aut (X)$-translates of finitely many rational polyhedral subcones $g_{j}(\Pi_{j})$ $(j\in J_{X})$.
\end{lemma}

\begin{proof}
The proof is the same as in \cite[Lemma 2.6, Corollary 2.7]{Markman2015}, but we include the proof.
Suppose that $g, h \in \Bir (X)$ satisfy $g(\Pi_{j}),h(\Pi_{j}) \subset \nef_{X}$.
Take a class $\alpha$ in the interior of $\Pi_{j}$.
Then $g(\alpha)$ and $h(\alpha)$ is an ample class, so $gh^{-1} \in \Aut (X)$. Thus we have the first assertion.
By Lemma \ref{lemnef}, the cone $\nef_{Y}^{+}$ is the union of $\Bir (X)$-translates of the $\Pi_{j}$ intersecting its interior.
By the first assertion, this is a union of $\Aut (X)$-translate of $g_{j} (\Pi_{j})$. Therefore, we have the second assertion.
\end{proof}

Now we have the following main theorem of this subsection.

\begin{theorem}\label{automcone}
Let $k$ be a field of characteristic $0$.
Let $X$ be am \hyp variety over $k$.
Suppose that $b_{2}(\barX) \geq 5$.
Then the following hold.
\begin{enumerate}
\item
The set $\mathfrak{B}_{X}$ of $k$-isomorphism classes of \hyp varieties in the $k$-birational class of $X$ is finite.
\item
There exists a rational polyhedral cone $D$ in $\nef_{X}^{+}$ which is a fundamental domain for the action of $\Gamma_{\Aut (X)}$. Here, $\Gamma_{\Aut(X)}$ means the image of the automorphism group of $X$ via the natural action $\Aut (X) \rightarrow \algO (\Lambda_{X, \R}).$ 
\end{enumerate}
\end{theorem}

\begin{proof}
(1) follows from Lemma \ref{birJ}, since the map
$\mathfrak{B}_{X} \rightarrow 2^{J}$,\ $Y \mapsto J_{Y}$ induces an injection.

(2) 
We choose an ample class $y \in \pos_{X} \cap \Lambda.$ 
We put
\[
D \coloneqq 
\{
x \in \nef_{X} \mid (\gamma x, y) \geq (x,y) \textup{ for all } \gamma \in \Gamma_{\Aut (X)}
\}.
\]
Then $D$ is a rational polyhedral fundamental domain by the proof of \cite[Lemma 2.2]{Totaro2010} and Lemma \ref{lemautcoset}.
\end{proof}

\begin{remark}\label{remautomcone}
Theorem \ref{automcone} (2) also holds for the case where $b_{2} (\barX) \leq 4$.
Indeed, if $\Lambda_{X}$ is of rank $1$, the cone conjecture is trivial.
On the other hand, if $\Lambda_{X}$ is of rank $2$, Then $\Lambda_{\barX} = \Lambda_{X}$, thus by the argument below \cite[Theorem 5.6]{Amerik2017}, the cone conjecture also holds.
\end{remark}
\subsection{Proof of the main theorem}
First, we prove an analogue of Lemma \ref{weylK3}.

\begin{lemma}
\label{lemmamainrefined}
Let $k$ be a field of characteristic $0$, and $X$ an \hyp variety over $k$ with $b_{2}(X_{\overline{k}}) \geq 5$.
Then there exists a positive integer $d$ 
such that there exists a polarization $L_{Y}$ on $Y$ with $(L_{Y},L_{Y}) = d$ for  any \hyp variety $Y$ satisfying the following conditions.
\begin{enumerate}
\item
There exists a finitely generated subfield $K \subset k$ field embeddings $\iota_{1}, \iota_{2}\colon K \hookrightarrow \C$, and \hyp variety $X', Y'$ over $K$ with $X'_{k} \simeq X$ and $Y'_{k} \simeq Y$, such that the complex manifolds $X_{\iota_{1},\C}$ and $Y_{\iota_{2},\C}$ are homeomorphic.
\item
There exits an isometry $\Phi\colon \Lambda_{X_{\overline{k}}} \simeq \Lambda_{Y_{\overline{k}}}$ which induces $\Lambda_{X} \simeq \Lambda_{Y}$. 
\end{enumerate}
\end{lemma}

\begin{proof}
By \cite[Theorem 5.3]{Amerik2017} and \cite[Corollary 1.4]{Amerik2016}, there exists a positive number $N$ such that any primitive MBM class on $Y_{\overline{k}}$ has the Beauville--Bogomolov square greater than $-N$, for any \hyp variety $Y$ over $k$ satisfying the condition (1).
We put
 \[
 \algO^{+}(\Lambda_{X_{\overline{k}}}, \Lambda_{X})
 :=
 \{
 g \in \algO^{+}(\Lambda_{X_{\overline{k}}}) \mid g\Lambda_{X} = \Lambda_{X}
\}.
 \] 
 We also put
 \[
 S:= \pos_{X}^{+} \setminus
 \left(
 \left\{
v^\perp \left|
\begin{array}l
 v \in \Lambda_{X_{\overline{k}}},
(v,v)>-N,\\
v^{\perp} \nsupseteq \Lambda_{X}
\end{array}
\right.
\right\}
\cup \bigcup_{\substack{g\in \algO^{+}(\Lambda_{X_{\overline{k}}}, \Lambda_{X})\\
\textup{s.t.\,} 1 \neq \overline{g} \in \algO^{+}(\Lambda_{X})
}} \mathrm{Fix}(g)
\right)
\]
Then there is a natural action $\algO^{+}(\Lambda_{X_{\overline{k}}}, \Lambda_{X})$ on $\pi_{0}(S)$.
We shall see that there exist only finitely many orbits of this action.
Since the image of $\algO^{+}(\Lambda_{X_{\overline{k}}}, \Lambda_{X})$ in $\algO^{+}(\Lambda_{X})$ is of finite index by \cite[Proposition 2.2]{Bright2019}, there exists a rational polyhedral fundamental domain $\mathcal{P} \subset \pos_{X}^{+}$ of the action of $\algO^{+}(\Lambda_{X_{\overline{k}}}, \Lambda_{X})$ on $\Lambda_{X}$.
By \cite[Proposition 3.4]{Markman2015} and \cite[Chapter 8, Remark 2.2]{Huybrechts2016}, 
\[
\mathcal{P} \setminus 
\left(
 \left\{
v^\perp \left|
\begin{array}l
 v \in \Lambda_{X_{\overline{k}}},
(v,v)>-N,\\
v^{\perp} \nsupseteq \Lambda_{X}
\end{array}
\right.
\right\}
\cup \bigcup_{\substack{g\in \algO^{+}(\Lambda_{X_{\overline{k}}}, \Lambda_{X})\\
\textup{s.t.\,}1 \neq \overline{g} \in \algO^{+}(\Lambda_{X})
}} \mathrm{Fix}(g)
\right)
\]
has only finitely many connected components $W_{1}, \ldots, W_{m}$.
There exists a connected component $V_{1}, \ldots, V_{m}$ of $S$ such that $W_{i} \subset V_{i}$ for $1 \leq i \leq m$.
Then we can show that $V_{1}, \ldots, V_{m}$ represent all the $\algO^{+}(\Lambda_{X_{\overline{k}}}, \Lambda_{X})$-orbits of  $\pi_{0}(S)$.
Indeed, for any connected component $V$ of $S$, there exists an element $g \in \algO^{+}(\Lambda_{X_{\overline{k}}})$ such that $gV \cap \mathcal{P}^{\circ} \neq \phi$. We have $gV \cap W_{j} \neq \phi$ for some $j$, and then we have $V_{j} = gV$.

Next, take an element $L_{i} \in V_{i} \cap \Lambda_{X} \cap \pos_{X}$.
We put $d_{i}:= (L_{i},L_{i})$, and $d:= \max_{1\leq i \leq m}d_{i}$.
Take an \hyp variety $Y$ over $k$, with an isometry $\Phi\colon \Lambda_{X_{\overline{k}}} \simeq \Lambda_{Y_{\overline{k}}}$ which induces $\Lambda_{X} \simeq \Lambda_{Y}$.
We may assume that $\Phi (\pos_{X}) = \pos_{Y}$.
By Proposition \ref{fundweylrational}, Definition \ref{defbiramp} and Proposition \ref{biramp}, the boundary $\partial(\Phi^{-1} (\amp_{Y})) \cap \pos_{X}$ is contained in 
\[
\left(
 \left\{
v^\perp \left|
\begin{array}l
 v \in \Lambda_{X_{\overline{k}}},
(v,v)>-N,\\
v^{\perp} \nsupseteq \Lambda_{X}
\end{array}
\right.
\right\}
\cup \bigcup_{\substack{g\in \algO^{+}(\Lambda_{X_{\overline{k}}}, \Lambda_{X})\\
\textup{s.t.\,}1 \neq \overline{g} \in \algO^{+}(\Lambda_{X})
}} \mathrm{Fix}(g)
\right)
\]
Therefore, there exists a connected component $V$ of $S$ such that $V \subset \Phi^{-1} (\amp_{Y})$.
There exists an element $g \in \algO^{+}(\Lambda_{X_{\overline{k}}}, \Lambda_{X})$ and $1\leq i \leq m$ such that $gV =V_{i}$.
Then $L_{i,Y}:=\Phi(g^{-1}L_{i})$ gives an ample line bundle on Y with $(L_{i,Y},L_{i,Y}) =d_{i} \leq d$.
Replacing $d$ by $d!$, it finishes the proof.
\end{proof}

\begin{remark}
\label{remarkmatsushitazhang}
Let $X$ be an \hyp variety over $k$ with a polarization of Beauville--Bogomolov square $d$, and $Y$ an \hyp variety over $k$ such that there exists an isometry $\Lambda_{X_{\overline{k}}} \simeq \Lambda_{Y_{\overline{k}}}$ which induces $\Lambda_{X} \simeq \Lambda_{Y}$.
Then $Y$ does not necessarily admit a polarization of Beauville--Bogomolov square $d$.
However, by using the action of $R_{Y}$ (see Proposition \ref{fundweylrational}), there exists an element $L \in \overline{\mv}_{Y} \cap \pos_{Y}$.
Therefore, by applying \cite[Section 4]{Matsushita2013}, there exists an \hyp variety $Z$ over $k$ which is birational to $Y$, and $Z$ admits a quasi-polarization of Beauville--Bogomolov square $d$.
(Note that though \cite{Matsushita2013} only treats the case of $k=\C$, the same result over general $k$ holds since we can run a similar MMP as in \cite{Matsushita2013} in our case.
Indeed, the cone theorem and the contraction theorem hold by the same proof as in \cite[Theorem 3.7]{Kollar1998}, and the existence of flips can be reduced to $\C$-case. Moreover, the termination of flips also can be proved by the same argument as in \cite[Section 4]{Matsushita2013}, and the irreducible symplecticness of output variety can be checked after the base change to $\C$, by the same argument as in \cite[Section 4]{Matsushita2013}.) 
\end{remark}

\begin{lemma}
\label{lemmamainrefined2}
Let $k$ be a field of characteristic $0$, $k'$ an algebraic field extension of $k$, and $X$ an \hyp variety over $k$ with $b_{2}(X_{\overline{k}}) \geq 5$..
Then there exists a positive number $d$ satisfying the following.
For any \hyp variety $Y$ over $k$ such that $Y_{k'}$ is birational to $X_{k'}$, there exists a polarization of Beauville--Bogomolov square $d$ on $Y$.
\end{lemma}

\begin{proof}
We may assume that $k'$ is an algebraic closure of $k$.
By the assumption, we have a $\Z$-lattice isometry between $\Lambda_{X_{k'}}$ and $\Lambda_{Y_{k'}}$. 
As in the argument in Lemma \ref{weylK3}, by \cite[Section 5, (a)]{Borel1963}, the candidates of 
isomorphism classes of $\Gal(k'/k)$-$\Z$-lattice $\Lambda_{Y_{k'}}$ is only finitely many.
Therefore, the desired statement follows from Lemma \ref{lemmamainrefined} and \cite[Theorem 4.6]{Huybrechts1999}.
\end{proof}

Next, we shall prove the following analogue of Lemma \ref{finquasipol} as a Corollary of cone conjectures.

\begin{lemma}\label{hypfinquasipol}
Let $k$ be a field of characteristic $0$.
Let $X$ be an \hyp variety over $k$.
Fix a positive integer $d$.
Then the set of quasi-polarizations on $X$ of Beauville--Bogomolov square $d$ modulo $\Aut (X)$ is a finite set.
\end{lemma}
\begin{proof}
This follows from the automorphism cone conjecture and the same argument as in Lemma \ref{finquasipol}.
\end{proof}

The following lemma is used for treating $b_{2} (\barX) \leq 4$.

\begin{lemma}
\label{lemmapic2}
Let $k$ be a  field of characteristic $0$.
Let $X$ be an \hyp variety over $k$.
Suppose that the rank of $\Lambda_{\barX}$ is $2$.
Then the automorphism group $\Aut (X)$ is either a finite group or an almost finite group of rank $1$ (see Definition \ref{almab}).
\end{lemma}

\begin{proof}
Since $\Aut (X) \subset \Aut (\barX)$, we may assume that $k$ is algebraically closed.
We note that the natural morphism
$\rho\colon \Aut (X) \rightarrow \algO (\Lambda_{X})$ has finite kernel.
Then applying \cite[Proposition 8.3 (2)]{Oguiso2008} to an exact sequence
\[
1 \rightarrow \ker \rho \rightarrow \Aut (X) \rightarrow \Ima \rho \rightarrow 1,
\]
we can reduce the problem to $\Ima \rho$.
Moreover, applying \cite[Proposition 8.3 (1)]{Oguiso2008} to an exact sequence
\[
1 \rightarrow \Ima \rho \cap \SO (\Lambda_{X}) \rightarrow \Ima \rho \rightarrow \Ima \det \rightarrow 1,
\]
we can reduce the problem to $\Ima \rho \cap \SO (\Lambda_{X})$.
On the other hand, $\SO (\Lambda_{X})$ is a finite group or an almost finite group of rank $1$ by \cite[Theorem 87]{Dickson1957}. Thus it finishes the proof.
\end{proof}

In this section, we prove the following theorem.
\begin{theorem}\label{hypkfintwist}
Let $k$ be a field of characteristic $0$.
Let $k' /k$ be a finite extension of fields.
Let $X$ be an \hyp variety over $k$.
Then the set $\Tw_{k'/k} (X)$ is finite.
\end{theorem}

\begin{proof}
Taking the Galois closure, we may assume that $k'/k$ be a finite Galois extension.
First, we consider the case where $b_{2}(\barX) \geq 5$.
Let $d$ be a positive integer as in Lemma \ref{lemmamainrefined2}.
Then any $Y \in \Tw_{k'/k}(X)$ admits a polarization $M_{Y}$ of Beauville--Bogomolov square $d$.
Moreover, by Lemma \ref{hypfinquasipol}, we can take a complete system of representatives $M_{1}, \ldots M_{m} \in \Lambda_{X_{k'}}$ of polarizations on $X_{k'}$ of Beauville--Bogomolov square $d$ modulo $\Aut (X_{k'})$.

For $1\leq i \leq m$, we put
\[
T_{i} \coloneqq 
\left\{
(Y, M) \left| 
\begin{array}{l}
Y \colon \textup{\hyp variety over } k \\
M \colon \textup{polarization on } Y\\
(Y,M)_{k'} \simeq (X_{k'}, M_{i})
\end{array}
\right.\right\}/k\textup{-isom}.
\]
Then we have an injective morphism of sets
\[
\Tw_{k'/k}(X) \rightarrow \bigsqcup_{i} T_{i},\ Y \mapsto (Y,M_{Y}),
\]
where $M_{Y}$ is taken as the above argument.
The finiteness of $T_{i}$ follows from the finiteness of the automorphism group of polarized \hyp variety (\cite[Proposition 2.26]{Brion2018}).
Therefore, $\Tw_{k'/k}(X)$ is a finite set.

Next, we treat the case of $b_{2}(\barX) \leq 4$.
In this case, the Picard rank of $\barX$ is $1$ or $2$.
If the Picard rank of $\barX$ is $1$, then the finiteness is reduced to the polarized case.
If the Picard rank of $\barX$ is $2$, by Lemma \ref{lemmapic2} and Proposition \ref{propgalcoh}, $\Tw_{k'/k} (X)$ is also finite. It finishes the proof of Theorem \ref{hypkfintwist}.
\end{proof}

\section{Uniform bounds}
\label{Uniformbounds}
In this section, we prove the following uniformness result for $\Tw_{k'/k}(X)$.

\begin{theorem}\label{effective}
Let $k'/k$ be a finite extension of fields of characteristic $p$.
\begin{enumerate}
\item
In Theorem \ref{fintwistK3}, the cardinality of the set $\Tw_{k'/k}(X)$ is bounded above by a constant which depends only on $p$, $[k'\colon k]$ and $\disc (\Lambda_{\barX})$.
\item
Suppose that $k$ is a subfield of $\C$.
Let $X$ be an \hyp variety over $k$.
Then the cardinality of the set $\Tw_{k'/k} (X)$ 
is bounded above by a constant which depends only on $[k'\colon k]$, $\disc (\Lambda_{\barX}, q)$ and the deformation class of $X_{\C}$.
Here, we say complex manifolds $Y_{1}$ and $Y_{2}$ are deformation equivalent if there exist connected complex spaces $\mathcal{Y}$ and $S$, a smooth proper holomorphic map $\pi \colon \mathcal{Y} \rightarrow S$, and two points $s_{1}, s_{2}$ such that $Y_{i}$ is biholomorphic to $\pi^{-1} (s_{i})$.
\end{enumerate}
\end{theorem}

Throughout the proof, we note that we may fix the lattice isometry class of $\Lambda_{\barX}$ since the set of isometry classes of $\Z$-lattices with bounded rank and discriminant is finite (\cite[$\ch$. 9, Theorem 1.1]{Cassels1982}).
The proof is accomplished by seeing constants appearing in the proof of Theorem \ref{hypkfintwist} carefully. 

\subsection{Bounds of automorphism groups}
In this subsection, we establish a uniform bound of canonical quasi-polarized variety, based on the work of Kov\'{a}cs. 

\begin{prop}\label{boundaut}
Let $k$ be an algebraically closed field of characteristic $0$.
Let $X$ be a smooth projective $n$-dimensional variety over $k$ with the trivial canonical bundle.
Let $L$ be a nef big line bundle on $X$.
Then the order of the group $\Bir (X,L)$ 
is finite and bounded above by a constant which depends only on $n$ and $L^{n}$.
\end{prop}

\begin{proof}
First, by the effective base point free theorem \cite[Theorem 1.1]{Kollar1993},
$2(n+2)!nL$ is base point free. We put $m \coloneqq 2(n+2)!n$.
Since $\Bir (X,L) \subset \Bir (X, mL)$, we shall establish a bound of $\Bir (X, mL)$.
We denote the natural action on the linear system $|mL|$ by $\rho \colon \Bir (X,mL)  \rightarrow \PGL (H^{0}(X,mL))$.
First, we shall bound the order of cyclic subgroups of $\Bir (X,mL)$.
Let $\langle g \rangle = G \subset \Bir (X,mL)$ be a cyclic subgroup.
Let $\widetilde{g} \in \GL (H^{0}(X,mL))$ be a lift of $\rho (g)$.
Let $v \in H^{0}(X,mL)$ be an eigenvector of $\widetilde{g}$.
Then the corresponding effective divisor $D_{v} \in |mL|$ satisfies $g_{\ast} (D_{v}) = D_{v}$, i.e.\ $G \subset \Bir (X,D_{v})$.
Here, we denote the set of automorphisms of $X$ preserving the Cartier divisor $D$ by $\Aut (X,D)$.
Therefore, by \cite[Theorem 1.2.3]{Kovacs2001}, we have 
\[
\# G \leq (m^{n}L^{n})^{n+1}.
\]
Note that in \cite{Kovacs2001}, it is supposed that $(X, D_{v})$ is log canonically polarized, i.e.\  $D_{v}$ is ample. But the same proof works in our situation. 
Now we note that the contraction $\phi_{mL} \colon X \rightarrow Y$ induces the inclusion $\Aut (X, mL) \rightarrow \Aut (Y, \oo(1))$, in particular $\Aut (X, mL)$ acts on a degree $m^{n}L^{n}$ projective subvariety $Y$ faithfully.
Therefore, by \cite[main bound]{Szabo1996}, we have a bound 
\[
\# \Aut (X, mL) \leq (m^{n} L^{n})^{16n3^{n}}.
\]
Therefore, it finishes the proof.
\end{proof}
Note that we immediately have the following corollary.
\begin{corollary}
\label{corboundaut}
Let $k, X, n, L$ be as in Proposition \ref{boundaut}.
Let $k'/k$ be a finite extension of fields of characteristic $0$.
Then $\#\Tw_{k'/k} (X,L)$ is bounded above by a constant which depends only on $n$, $L^{n}$, and $[k'\colon k]$.
\end{corollary}

\begin{remark}\label{boundautK3}
Let $X$ be a K3 surface over an algebraically closed field $k$.
Since $\Aut (X) \rightarrow \algO (H^{2}_{\et}(X,\Q_{\ell}))$ is injective for any $\ell \neq p$ by \cite[Corollary 2.5]{Ogus1979} and \cite[Theorem 1.4]{Keum2016}, the order of any finite subgroup of $\Aut (X)$ is bounded by $4^{22^{2}}$.
Here, we use that $1 + n M_{22} (\Z_{\ell}) \subset \GL_{22} (\Z_{\ell})$ is torsion free if $n = \ell ^{2} =4$ or $ n= \ell \geq 3$.
\end{remark}

\subsection{Proof of Theorem \ref{effective} (1)}
In this subsection, we follow the notation in Theorem \ref{effective} (1).
For simplicity, if some real number $c$ admits an upper bound which depends only on $p, [k' \colon k]$, and $\disc (\Lambda_{\barX})$, then we say that $c$ is uniformly bounded.

To begin with, as in the beginning of the proof of Theorem \ref{fintwistK3}, we may assume that $k'$ is a finite separable extension.
Moreover, we may assume that $\Lambda_{X_{k'}} =  \Lambda_{\sepX} = \Lambda_{\barX}$.
Indeed, an order of a torsion subgroup of $\GL (\Lambda_{\sepX})$ is bounded by $\# \GL_{\rho} (\F_{3})$, since $1 + 3 \mathrm{M}_{\rho}(\Z_{3}) \subset \GL_{\rho} (\Z_{3})$ is a torsion free subgroup. 

Next, we shall show that we may assume $\Aut (X_{k'}) = \Aut (\sepX) (= \Aut (\barX))$.
We need the following lemma.
Let $\rho \colon \Aut (\barX) \ltimes R_{\barX} \rightarrow \algO (\Lambda_{\barX})$
be the natural morphism. Here, we write $R_{\barX}$ for  the subgroup of $\algO (\Lambda_{\barX})$ generated by reflections by $-2$ classes as in \cite{Bright2019}.
We note that $R_{\barX}$ is equal to $R_{\sepX}$ which is defined similarly, by \cite[Corollary 3.2]{Bright2019}. Therefore, the absolute Galois group $G_{k}$ acts on $\Aut(\barX)$, $R_{\barX}$, or $\algO (\Lambda_{\barX})$.
\begin{lemma}\label{index}
The index of $\Ima \rho$ in $\algO (\Lambda_{\sepX})$ is bounded above by a constant which depends only on the lattice isometry class of $\Lambda_{\sepX}=\Lambda_{\barX}$.
\end{lemma}

\begin{proof}
In the characteristic $0$ case, this follows from the proof of \cite[Chapter 8, Theorem 4.2]{Huybrechts2016}.
Indeed, any element in $\ker (\algO (\Lambda_{\barX}) \rightarrow \algO (\Lambda_{\barX}^{\ast} / \Lambda_{\barX}))$ comes from the image of $\rho$.
We treat the positive characteristic case.
If $X$ is supersingular, then this finiteness of the index essentially follows from the crystalline Torelli theorem \cite{Bragg2018}.
In this case, by the proof of \cite[Proposition 5.2]{Lieblich2018}, the index $[\algO (\Lambda_{\barX}) \colon \rho (\Aut (X_{\barX}) \ltimes R_{\barX})]$ is absolutely bounded. Indeed, the index 
\[
[\algO^{+} (\Lambda_{\barX}) / R_{\barX}\colon \Ima (\Aut (\barX) \rightarrow \algO^{+} (\Lambda_{\barX}) / R_{\barX})]
\]
is at most $\# \GL (\Lambda_{\barX} \otimes_{\Z} \F_{p})$.

On the other hand, if $X$ has finite height, there exists a projective K3 families $\mathcal{X} \rightarrow \W (\overline{k})$ with $\mathcal{X}_{\overline{k}} \simeq X_{\overline{k}}$ such that the restriction map $\Pic (\mathcal{X}) \rightarrow \Pic (\barX)$ is an isomorphism. 
In this case, we have $\Lambda_{\mathcal{X}_{\overline{F}}} \simeq \Pic (\mathcal{X}) \simeq \Lambda_{\barX}$. 
Here, we put $F \coloneqq \Frac (\W (\overline{k}))$.
Therefore, by the case of $\chara k = 0$, the index 
\[
[\Lambda_{\mathcal{X}_{\overline{F}}} \colon \Ima (\Aut (\mathcal{X}_{\overline{F}}) \ltimes R_{\mathcal{X}_{\overline{F}}} \rightarrow \algO (\Lambda_{\mathcal{X}_{\overline{F}}}))]
\]
is uniformly bounded.
Since the map
\[
\Aut (\mathcal{X}_{\overline{F}}) \ltimes R_{\mathcal{X}_{\overline{F}}} \rightarrow \algO (\Lambda_{\mathcal{X}_{\overline{F}}}) \simeq \algO (\Lambda_{\barX})
\]
factors through the specialization map $\Aut (\mathcal{X}_{\overline{F}}) \ltimes R_{\mathcal{X}_{\overline{F}}} \rightarrow  \Aut (\barX) \ltimes R_{\barX}$ by \cite[Theorem 2.1]{Lieblich2018}, it finishes the proof.
\end{proof}

\begin{lemma}
The number of generators of the image of $\rho$ is uniformly bounded.
In particular, the number of generators of $\Aut ( \barX) \ltimes R_{\barX}$ and $\Aut (\barX)$ are uniformly bounded.
\end{lemma}
\begin{proof}
The first statement follows from Lemma \ref{index} and the Schreier index formula.
The second statement follows from the first statement since the order of $\ker \rho$ is uniformly bounded by Remark \ref{boundautK3}.
\end{proof}

Since the order of $\ker \rho$ is uniformly bounded by Remark \ref{boundautK3}, we may assume that $\ker \rho$ is $G_{k'}$-trivial.
Let $\gamma_{1}, \ldots, \gamma_{N}$ be generators of $\Aut (\sepX)$. For any $\sigma \in G_{k'}$, we have $\sigma (\gamma_{i}) = \gamma_{i} \delta_{i,\sigma}$ for $\delta_{i,\sigma} \in \ker \rho$.
Then the map $\sigma \mapsto \delta_{i, \sigma}$ gives a group morphism $\psi_{i}\colon G_{k'} \rightarrow \ker \rho$.
We note that the index of $\cap_{i} \ker \psi_{i}$ in $G_{k'}$ is uniformly bounded, since so are $N$ and the order of $\ker \rho$.
Replacing $k'$ by a finite field extension corresponding to $\cap_{i} \ker \psi_{i}$, we may assume that $\Aut (\sepX) = \Aut (X_{k'}).$

We need the following lemma to bound the polarization degree uniformly.

\begin{lemma}
\label{effectiveweylK3}
In the statement of Lemma \ref{weylK3}, 
we can take a positive integer $d$ which depends only on the lattice isometry class of $\Lambda_{\barX}$ (does not depend on $[k'\colon k]$).
\end{lemma}

\begin{proof}
The number of candidates of isomorphism classes of $\Z$-lattices $\Lambda_{Y}$ in Lemma \ref{weylK3} depend on the lattice isometry class of $\Lambda_{\barX}$.
On the other hand. as in the proof of Lemma \ref{weylK3}, the set of degree of polarizations on $Y$ depends only on the lattice isometry class of $\Lambda_{Y}$. Therefore, it finishes the proof.
\end{proof}

In the proof of Theorem \ref{fintwistK3},
first, we took a complete system of representatives $M_{1}, \ldots, M_{m} \in \Lambda_{X_{k'}}$ of polarizations on $X_{k'}$ of degree $d$. Here, $d$ is taken as in Lemma \ref{effectiveweylK3}.
Next, we defined the injective map $\Tw_{k'/k} (X) \rightarrow \bigsqcup_{1 \leq i \leq m} T_{i},\ Y\mapsto (Y,L_{Y})$.
Here, $T_{i}$ is a set of isomorphism classes of $k$-forms of $(X_{k'}, M_{i})$ (see the proof of Theorem \ref{fintwistK3}).
Therefore, it is enough to show that the integer $m$ and the cardinalities of $T_{i}$ are uniformly bounded.
Note that the cardinalities of $T_{i}$ are bounded by a constant which depends only on $[k'\colon k]$ by Remark \ref{boundautK3}. We consider the bound of the integer $m$.
Since $\nef (\barX)$ is a fundamental domain with respect to the action of $R_{\barX}$, the problem can be reduced to bound the cardinality of the following set
\[
\{
\lambda \in \Lambda_{\barX} \mid
(\lambda, \lambda) = d
\}
/ \Aut (\barX) \ltimes R_{\barX}.
\]
We may fix the lattice isometry class of $\Lambda_{X_{\overline{k}}}$.
For any lattice isometry class $\Lambda$, one can fix a rational polyhedral fundamental domain $D_{\Lambda} \subset \mathcal{C}_{\Lambda}^{+}$ with respect to the action of $\algO^{+}(\Lambda)$. Here, $\mathcal{C}_{\Lambda}^{+}$ is the convex hull of $\overline{\mathcal{C}_{\Lambda}} \cap \Lambda_{\Q}$, and $\mathcal{C}_{\Lambda}$ is the positive cone of $\Lambda$.
Moreover, we put  $\algO^{+} (\Lambda_{\barX}) = \bigcup_{l \in L} \Gamma g_{l}$, where $\Gamma \coloneqq \Ima \rho$.
We note that the cardinality of the index set $L$ is uniformly bounded by Lemma \ref{index}.
Then $\bigcup_{l\in L} g_{l} D_{\Lambda_{\barX}}$ is a fundamental domain with respect to the action of $\Aut (\barX) \ltimes R_{\barX}$.
Therefore, we have the desired uniform boundedness.

\subsection{Proof of Theorem \ref{effective} (2)}
In this subsection, we follow the notation in Theorem \ref{effective} (2).
For simplicity, if some real number admits upper bound which depends only on $[k'\colon k]$, $\disc (\Lambda_{\barX},q)$, and the deformation class of $\barX$, then we say that $c$ is uniformly bounded.

To begin with, we may assume that $\Lambda_{X_{k'}} = \Lambda_{\barX}$.
Indeed, an order of a torsion subgroup of $\GL (\Lambda_{\barX})$ is bounded by $\# \GL(\Lambda_{\barX, \F_{3}})$, since $1 + 3 \mathrm{End}(\Lambda_{\barX, \Z_{3}}) \subset \GL (\Lambda_{\barX, \Z_{3}})$ is a torsion free subgroup. We note that the Picard number of $\barX$ is bounded above by the second Betti number of $\barX$, which depend only on a deformation class of $X$.
Moreover, we shall show that we may assume $\Bir (X_{k'}) = \Bir (\barX)$.
To prove this, we propose the following lemma.

\begin{lemma}\label{effectivelemmamain}
In the statement of Lemma \ref{lemmamainrefined2},
we can take a positive integer $d$ which depends only on the lattice isometry class of $\Lambda_{X_{\overline{k}}}$ and the deformation class of $X_{\C}$ (not on $[k'\colon k]$).
\end{lemma}
\begin{proof}
The number of candidates of isomorphism classes of $\Gal(k'/k)$-$\Z$-lattices $\Lambda_{Y_{\overline{k}}}$ in Lemma \ref{lemmamainrefined2} depends only on the lattice isometry class of $\Pic (\barX)$.
Moreover, the integer $N$ appearing in the proof of Lemma \ref{lemmamainrefined} depends only on the deformation class of $X_{\C}$. Therefore, the desired statement follows from the proof of Lemma \ref{lemmamainrefined} and Lemma \ref{lemmamainrefined2}.
\end{proof}
Take a positive integer $d$ in Lemma \ref{effectivelemmamain}.
Then $X$ admits a polarization $L$ of Beauville--Bogomolov square $d$.
As before, let $\rho \colon \Bir (\barX) \ltimes R_{\barX} \rightarrow \algO (\Lambda_{\barX})$ be the natural action.
The action of $\Gal(\overline{k}/k')$ on the right hand side is trivial, and $\ker \rho \subset \Bir (\barX)$ is contained in $\Bir (Y,L)$, whose order is uniformly bounded by Lemma \ref{boundaut}. 
Therefore, we may assume that $\ker \rho$ is $\Gal (\overline{k}/k')$-trivial.
\begin{lemma}\label{effectiveindex}
The index of $\Ima \rho$ in $\algO (\Lambda_{\barX})$ depends only on the lattice isometry class of $\Lambda_{\barX}$ and the deformation class of $X_{\C}$.
\end{lemma}
\begin{proof}
We may assume $\overline{k} = \C$.
As in the proof of \cite[Lemma 6.23]{Markman2011}, 
to prove Lemma \ref{effectiveindex}, it suffices to bound the following.
\begin{enumerate}
\item
The index of $\Ima (\algO^{+}_{\mathrm{Hdg}}(H^{2}(X_{\C},\Z)) \rightarrow \algO^{+}(\Lambda_{X_{\C}}))$ in $\algO^{+} (\Lambda_{X_{\C}})$.
\item
The index of $\mathrm{Mon}^{2}(X_{\C})$ in  $\algO^{+}(H^{2} (X_{\C},\Z))$.
\end{enumerate}
Here, $\algO^{+}_{\mathrm{Hdg}}(H^{2}(X_{\C},\Z)) \subset \algO^{+}(H^{2}(X_{\C},\Z))$ is the subgroups of elements which preserve the Hodge structure, and $\mathrm{Mon}^{2}(X_{\C})$ is the monodromy group restricted on the second cohomology (see \cite[Definition 1.1]{Markman2011}).
$(1)$ depends on the lattice isometry class of $\disc \Lambda_{\barX}$ (see \cite[Chapter 14, Proposition 2.6]{Huybrechts2016}).
$(2)$ clearly depends on the deformation class of $X_{\C}$. Therefore, it finishes the proof.
\end{proof}

\begin{lemma}
The number of generators of the image of $\rho$ is uniformly bounded.
In particular, the number of generators of $\Bir (\barX) \ltimes R_{\barX}$ and $\Bir (\barX)$ are uniformly bounded.
\end{lemma}
\begin{proof}
The first statement follows from Lemma \ref{effectiveindex} and the Schreier index formula.
The second statement follows from the first statement since the order of $\ker \rho$ is uniformly bounded.
\end{proof}
Let $\gamma_{1}, \ldots, \gamma_{N}$ be generators of $\Bir (\barX)$.
For any $\sigma \in \Gal (\overline{k}/k')$, we have $\sigma (\gamma_{i}) = \gamma_{i} \delta_{i, \sigma}$ for $\delta_{i,\sigma}\in \ker \rho$.
Then the map $\sigma \mapsto \delta_{i,\sigma}$ gives a group morphism $\psi_{i} \colon \Gal (\overline{k}/k') \rightarrow \ker \rho$. 
We note that the index of $\bigcap_{i} \ker \psi_{i}$ in $\Gal (\overline{k}/k')$ is uniformly bounded, since so are $N$ and the order of $\ker \rho$. 
Replacing $k'$ by a finite field extension corresponding to $\bigcap_{i} \ker \psi_{i}$, we may assume that $\Bir (\barX) = \Bir (X_{k'}).$

Next, we recall the argument in Theorem \ref{hypkfintwist}.
First, we shall consider the case of $b_{2} (\barX) \geq 5$.
We take a complete system of representatives $M_{1}, \ldots, M_{m} \in \Lambda_{X_{k'}}$ of polarizations on $X_{k'}$ of Beauville--Bogomolov square $d$. Here, $d$ is taken as in Lemma \ref{effectivelemmamain}.
We have an injective morphism of sets $\Psi \colon \Tw_{k'/k}(X) \rightarrow \bigsqcup_{i} T_{i}$, where $T_{i}$ is the set of twists of $(X_{k'},M_{i})$ for $k'/k$.
Therefore, we should bound the following constants.
\begin{enumerate}
\item
The cardinality of each set $T^{i}$.
\item
The positive integers $m$
\end{enumerate}

\subsection*{Bound of (1)}
This follows from Proposition \ref{boundaut}.

\subsection*{Bound of (2)}
In the following, we use the same notation $D_{\Lambda}$ as in the proof of Theorem \ref{effective} (1).
Since a birational automorphism which sends a polarization to a polarization is an automorphism, we need to bound the set of degree $d$ elements $\lambda \in \overline{\mv}(\barX)^{\circ} \cap \Lambda_{\barX}$ modulo birational automorphisms.
Since $\overline{\mv}(\barX)$ is a fundamental domain with respect to the action of $R_{\barX}$, the problem can be reduced to bound the cardinality of the following set.
\[
\{
\lambda \in \Lambda_{\barX} \mid
(\lambda, \lambda) = d
\}
/ \Bir (\barX) \ltimes R_{\barX}.
\]
Let $\rho \colon \Bir (\barX) \ltimes R_{\barX} \rightarrow \algO(\Lambda_{\barX})$ be the natural map and $\Gamma$ the image of $\rho$.
Moreover, we put  $\algO^{+} (\Lambda_{\barX}) = \bigcup_{l \in L} \Gamma g_{l}$.
We note that the cardinality of the index set $L$ is uniformly bounded by Lemma \ref{effectiveindex}.
Then $\bigcup_{l\in L}g_{l} D_{\Lambda_{\barX}}$ is a fundamental domain with respect to the action of $\Bir(\barX) \ltimes R_{\barX}$, and we have the desired uniform boundedness.

Finally, we consider the case where $b_{2} (\barX) \leq 4$.
If the Picard rank of $\barX$ is 1, then we have desired finiteness by Proposition \ref{boundaut}, since we may fix the lattice isometry class of $\Lambda_{\barX}$.
If the Picard rank of $\barX$ is 2, by Remark \ref{remarkgalcoh} and the proof of \cite[Proposition 8.3 (1), (2)]{Oguiso2008}, it is enough to bound $\# \ker \rho$ and the order of torsion part of $\SO (\Lambda_{\barX})$. 
The former is uniformly bounded as in the argument in the beginning of this subsection, and the latter is also uniformly bounded since we may fix the lattice isometry class of $\Lambda_{\barX}$.
Therefore, it finishes the proof.

\section{On the derived equivalent twists}
\label{Derivedequivalent}
In this section, we prove the finiteness of derived equivalent twists, as an application of the finiteness of twists for finite extensions.

\begin{prop}\label{FourierMukai}
Let $k$ be a field.
Let $X,Y$ be smooth projective varieties over $k$.
We denote their bounded derived categories of coherent sheaves by $D_{b}(X), D_{b}(Y)$.
Let p (resp,\, q) be the first projection $X\times_{k}Y \rightarrow X$ (resp.\,the second projection $X\times_{k} Y \rightarrow Y$).
Suppose that $X$ and $Y$ are derived equivalent, i.e.\  there exists 
an equivalence $F \colon D_{b} (X) \simeq D_{b}(Y)$.
Then there exists a perfect complex $P \in D_{b}(X\times_{k} Y)$ unique up to isomorphism, which is called the Fourier--Mukai kernel, such that $F$ is written as
\[
\Phi_{P}\colon D_{b} (X) \rightarrow D_{b}(Y),\ \varepsilon \mapsto Rq_{\ast}(Lp^{\ast} \varepsilon \otimes P).
\].
\end{prop}
\begin{proof}
This is Orlov's result (see \cite[Theorem 2.19]{Orlov} or \cite[Corollary 5.17]{Huybrechts2006}). 
\end{proof}

\begin{prop}\label{FMet}
Let $k$, $X$, $Y$, $p$, $q$, $P$ be as in Proposition \ref{FourierMukai}.
We put
\[
v (P) \coloneqq \ch (P)\cdot \sqrt{\td(X \times_{k} Y)},
\]
where $\ch (P)$ is the Chern character of $P$ and $\td(X \times_{k} Y)$ is the Todd class of the tangent bundle $\mathcal{T}_{X\times_{k}Y/k}$.
We put the cohomological Fourier--Mukai transform as
\[
\Phi_{P,\ell}^{\et}\colon H^{\ast}_{\et}(X_{\overline{k}},\Q_{\ell}) \rightarrow H^{\ast}_{\et}(Y_{\overline{k}}, \Q_{\ell}),\ \alpha \mapsto p_{\ast} (v(P) \cup q^{\ast}(\alpha)).
\]
Then $\Phi_{P,\ell}^{\et}$ gives a $\Gal(k_{s}/k)$-equivariant isomorphism of $\Q_{\ell}$-vector space.
Moreover, there exists an integer $N$ which depends on the dimension of $X$, such that for any $\ell > N$, the isomorphism $\Phi_{P,\ell}^{\et}$ induces an isomorphism $H^{\ast}_{\et}(X,\Z_{\ell}) \rightarrow H^{\ast}_{\et}(Y, \Z_{\ell})$.
\end{prop}
\begin{proof}
For the first statement, see \cite[Remark 5.30]{Huybrechts2006}, \cite[Section 2]{Lieblich2015}.
The second statement follows from the definition of the Chern character and the Todd class.
\end{proof}
\begin{remark}
As in \cite[Lemma 10.6]{Huybrechts2006}, in the case of K3 surfaces, the class $v(P)$ has an integral coefficient.
\end{remark}

\begin{definition}\label{level}
Let $k$ be a field of characteristic $0$, $X$ an \hyp variety over $k$, $\ell$ a prime number, and $n$ a positive number.
We denote the torsion-free part of $\bigoplus_{i} H^{i}_{\et}(X_{\overline{k}},\Z_{\ell})$ by $H^{\ast}_{\ell}(X_{\overline{k}})$.
Let $\la$ be the abstract $\Z_{\ell}$-lattice which is isomorphic to
$H^{\ast}_{\ell}(X_{\overline{k}})$
.
We define a \emph{ level $\ell^{n}$ structure of the full cohomology of 
$X$} 
as 
a $G_{k}$-invariant $\GL (\la, \ell^n)$-orbit $\alpha$ of an isomorphism over $\Z_{\ell}$ 
\[
\la \simeq H^{\ast}_{\ell}(X_{\overline{k}}).
\]
Here, we put $\GL (\la, \ell^n)$ as a principal level $\ell^{n}$-subgroup 
\[
\ker (\GL (\la) \rightarrow \GL (\la \otimes_{\Z_{\ell}} \Z_{\ell}/\ell^{n}\Z_{\ell})).
\]
We note that for any quasi-polarization $M$ of $X$, any automorphism of $(X,M,\alpha)$ acts trivially on $ H^{\ast}_{\ell}(X_{\overline{k}})$ if $\ell^{n} \geq3$, since $\Aut (X,M)$ is finite (see Proposition \ref{boundaut}) and $\GL(\la, \ell^{n})$ have no non-trivial torsion elements.
\end{definition}

\begin{theorem}\label{findertwist}
Let $k$ be a field.
Let $X$ be a smooth projective variety over $k$.
We put 
\[
\Tw^{D}(X) :=
\left\{ Y \colon\textup{variety over } k \left|
\begin{array}l
\barX \simeq_{\overline{k}} Y_{\overline{k}}, \\
X \textup{ is derived equivalent to } Y
\end{array}
\right.
\right\}/k\textup{-isom}.
\]
Then $\Tw^{D}(X)$ is finite in the following cases.
\begin{enumerate}
\item
$X$ is a K3 surface over $k$, and the characteristic of $k$ is not equal to $2$.
\item
$X$ is a K3 surface over $k$, and $X$ is not supersingular.
\item
$k$ is of characteristic $0$, and $X$ is \hyp variety over $k$ with the second Betti number $b_{2}(X_{\overline{k}}) \geq 5$ such that
\[
\Aut(X_{\overline{k}}) \rightarrow \Aut (H_{\et}^{\ast}(X_{\overline{k}},\Q_{\ell}))
\]
is injective for any prime number $\ell$.
\end{enumerate}
In particular, in the case of (1), the isomorphism classes of K3 surfaces over $k$ which are derived equivalent to $X$ are finitely many.
\end{theorem}

\begin{proof}
The last statement for the case (1) follows from the finiteness of $\Tw^{D}$ and the corresponding statement over algebraically closed field given by Bridgeland--Maciocia and Lieblich--Olsson (\cite[Corollary 1.2]{Bridgeland2001} and \cite[Theorem 1.1]{Lieblich2015}). 

In the following, we fix a prime number $\ell > \max\{N,2\}$, where $N$ is given by applying \ref{FMet} to $X$.

First, we prove the assertions (1) and (2).
Since the isomorphism functor is unramified (see the proof of Theorem \ref{fintwistK3}), we have
\[ 
\Tw^{D}(X) 
=
\left\{ Y \colon \textup{variety over } k \left|
\begin{array}l
X_{k_{s}} \simeq_{k_{s}} Y_{k_{s}}, \\
X \textup{ is derived equivalent to } Y
\end{array}
\right.
\right\}
\]
for any K3 surfaces $X$ over $k$. Here, we denote the separable closure of $k$ by $k_{s}$.
By Lemma \ref{weylK3}, there exists a positive integer $d$ such that every $Y \in \Tw^{D}(X)$ admit a polarization $M_{Y}$ of degree $d$.
Moreover, for any $Y_{1}, Y_{2} \in \Tw^{D} (X)$, we have a Galois equivariant isomorphism 
\[
\Phi_{P,\ell}^{\et} \colon H^{\ast}_{\et} (Y_{1,\overline{k}},\Z_{\ell}) \simeq H^{\ast}_{\et} (Y_{2,\overline{k}},\Z_{\ell})
\]
by the choice of $\ell$.
Therefore, we can take a finite Galois extension $k'/k$ such that $Y_{k'}$ admits a level $\ell$-structure $\alpha_{Y_{k'}}$ on the full cohomology for any $Y \in \Tw^{D}(X)$.
As in the proof of Theorem \ref{fintwistK3}, 
by Lemma \ref{finquasipol}, we can take a complete system of representatives $M_{1}, \ldots, M_{m} \in \Pic_{X/k}(k_{s})$ of polarizations of $X_{k_{s}}$ of degree $d$ modulo $\Aut _{X/k}(k_{s})$. 
Moreover, we denote possible level structures on the full cohomology of $(X_{k_{s}})$ by $(\alpha_{1},\ldots, \alpha_{M})$.
We put
\[
T_{i,j}
\coloneqq
\left\{
(Y,M,\alpha) \left|
\begin{array}l
Y \colon \textup{K3 surface over $k'$, }\\
M \colon \textup{polarization of $X$, }\\
\alpha \colon \textup{level } \ell\textup{-structure on the full cohomology of }Y, \\
(Y_{\overline{k}},M_{k_{s}},\alpha_{k_{s}}) \simeq_{\overline{k}} (X_{k_{s}},M_{i},\alpha_{j})
\end{array}
\right.
\right\}.
\]
Now we have a map of sets 
\[
\Tw^{D} (X) \rightarrow
\bigsqcup_{i,j} T_{i,j},\ Y \mapsto (Y_{k'},M_{Y,k'}, \alpha_{Y_{k'}}).
\]
This map has finite fiber by Theorem \ref{fintwistK3} in the cases $(1)$ and  $(2)$.
On the other hand, the automorphism group of a polarized K3 surface with a level $\ell$-structure on the full cohomology is trivial since the map $\Aut (\barX) \rightarrow H^{2}_{\et}(\barX, \Z_{\ell})$ is injective as in Remark \ref{boundautK3}  (see also Definition \ref{level}).
Therefore, each set $T_{i,j}$ is a singleton. Now we have the desired finiteness.

Next, we shall prove the assertion (3).
By Lemma \ref{lemmamainrefined2}, there exists a positive integer $d$ such that for any $Y \in \Tw^{D}(X)$, $Y$ admits a polarization $M_{Y}$ of Beauville--Bogomolov square $d$.
As before, we can take a finite Galois extension $k'/k$ such that $Y_{k'}$ admits a level $\ell$-structure $\alpha_{Y_{k'}}$ on the full cohomology for any $Y \in \Tw^{D}(X)$.

As in the proof of Theorem \ref{hypkfintwist}, 
by Lemma \ref{hypfinquasipol}, we can take a complete system of representatives $M_{1},\ldots, M_{m} \in \Lambda_{X_{\overline{k}}}$ of polarizations on $X_{\overline{k}}$ of Beauville--Bogomolov square $d$.
Moreover, we denote possible level $\ell$-structures on the full cohomology of $Y_{k'}$ by $\alpha_{1},\ldots, \alpha_{M}$.
We put
\[
T_{i,j}
\coloneqq
\left\{
(Y,M,\alpha) \left|
\begin{array}l
Y \colon \textup{\hyp variety over $k'$, }\\
M \colon \textup{polarization on $Y$, }\\
\alpha \colon \textup{level } \ell\textup{-structure on the full cohomology of }(Y,M), \\
(Y_{\overline{k}},M_{\overline{k}},\alpha_{\overline{k}}) \simeq_{\overline{k}} (X_{\overline{k}},M_{i},\alpha_{j})
\end{array}
\right.
\right\}.
\]
Now we have a map of sets
\[
\Tw^{D}(X) \rightarrow  \bigsqcup_{i,j} T_{i,j},\ Y \mapsto (Y_{k'}, M_{Y_{k'}}, \alpha_{Y_{k'}}).
\]
This map has finite fiber by Theorem \ref{hypkfintwist}.
On the other hand, 
by the assumption,
the automorphism group of a polarized \hyp variety with a level structure on the full cohomology is trivial (see the remark in Definition \ref{level}). 
Therefore, each set $T_{i,j}$ is a singleton, and we have the desired finiteness.
\end{proof}

\begin{corollary}
\label{corfindertwist}
Let $k$ be a field of characteristic $0$. and $X$ is \hyp variety over $k$.
Suppose that there exist a subfield $K\subset k$, an embedding $K \hookrightarrow \C$, and an \hyp variety $X'$ over $K$ with $X'_{k} \simeq X$ such that $X_{\C}$ is deformation equivalent to one of the following.
\begin{enumerate}
\item
The $n$-points Hilbert scheme $S^{[n]}$ of a K3 surface $S$ over $\C$ (in this case, we say that $X$ is of $K3^{[n]}$-type).
\item
The fiber of the summation morphism $A^{[n+1]} \rightarrow A$ over $0 \in A$, where $A^{[n+1]}$ is a $n+1$-points Hilbert scheme of an abelian surface over $\C$ (in this case, we say that $X$ is if generalized Kummer-type).
\item
O Grady's 6-dimensional varieties over $\C$ (see \cite{OGrady2003}) (in this case, we say that $X$ is $OG_{6}$-type).
\item
O Grady's 10-dimensional varieties over $\C$ (see \cite{OGrady1999}) (in this case, we say that $X$ is $OG_{10}$-type).
\end{enumerate}
Then $Tw^{D}(X)$ is a finite set.
In particular, in the case of (1), the isomorphism classes of $K3^{[n]}$-type varieties over $k$ which are derived equivalent to $X$ are finitely many.
\end{corollary}

\begin{proof}
The cases (1), (2), (4) follow from Theorem \ref{findertwist}, the argument in the proof of \cite[Theorem 1.3]{Oguiso2020}, \cite[Theorem 5.1]{Oguiso2020}, and \cite[Theorem 3.1]{Mongardi2017}.
For (3) follows from Theorem \ref{findertwist}, \cite[Remark 6.9]{Mongardi2017}, and a deformation argument (see \cite[Example 2.10 (3)]{Fu2022}).
The last statement follows from the Lefschetz principle and \cite[Theorem 9.4]{Beckmann2021}.
\end{proof}

\begin{remark}
\label{remfindertwist}
For a K3 surface $X$ over $k$, the map $\Aut (X) \rightarrow \GL(H^{2}_{\et}(X,\Z_{\ell}))$ is injective.
However, for general \hyp varieties over characteristic $0$ fields, this map is no longer injective (see \cite[Theorem 1.2]{Oguiso2020} and \cite[Theorem 5.2]{Mongardi2017}).
On the other hand,
as in Corollary \ref{corfindertwist}, the condition in Theorem \ref{findertwist}.(3) holds true for any known irreducible symplectic varieties.
\end{remark}

\newcommand{\etalchar}[1]{$^{#1}$}
\providecommand{\bysame}{\leavevmode\hbox to3em{\hrulefill}\thinspace}
\providecommand{\MR}{\relax\ifhmode\unskip\space\fi MR }
\providecommand{\MRhref}[2]{%
  \href{http://www.ams.org/mathscinet-getitem?mr=#1}{#2}
}
\providecommand{\href}[2]{#2}

\end{document}